\documentclass[12pt,a4paper]{amsart}

\PassOptionsToPackage{fleqn}{amsmath}

\usepackage[T1]{fontenc}
\usepackage{lmodern}

\DeclareFontFamily{U}{rsfs}{}
\DeclareFontShape{U}{rsfs}{m}{n}{<-> s*[1.0] rsfs10}{}

\makeatletter
\renewcommand\normalsize{%
    \@setfontsize\normalsize{9.7}{14pt plus .3pt minus .3pt}%
    \abovedisplayskip 10\p@ \@plus4\p@ \@minus4\p@
    \abovedisplayshortskip 6\p@ \@plus2\p@
    \belowdisplayshortskip 6\p@ \@plus2\p@
    \belowdisplayskip \abovedisplayskip}
\renewcommand\small{%
    \@setfontsize\small{9.5}{12\p@ plus .2\p@ minus .2\p@}%
    \abovedisplayskip 8.5\p@ \@plus4\p@ \@minus1\p@
    \belowdisplayskip \abovedisplayskip
    \abovedisplayshortskip \abovedisplayskip
    \belowdisplayshortskip \abovedisplayskip}
\renewcommand\footnotesize{%
    \@setfontsize\footnotesize{8.5}{9.25\p@ plus .1pt minus .1pt}
    \abovedisplayskip 6\p@ \@plus4\p@ \@minus1\p@
    \belowdisplayskip \abovedisplayskip
    \abovedisplayshortskip \abovedisplayskip
    \belowdisplayshortskip \abovedisplayskip}
\ifdefined\pdfpagewidth
\else
\fi
\calclayout
\makeatother

\usepackage{mathrsfs}
\usepackage{url}
\usepackage{graphicx}
\usepackage[toc,page]{appendix}
\usepackage{chngcntr}
\usepackage{etoolbox}
\usepackage{fullpage}
\usepackage{pgf,tikz,pgfplots}
\pgfplotsset{compat=1.15}
\usetikzlibrary{arrows}
\usetikzlibrary{positioning}
\usetikzlibrary{quotes}
\usepackage{tikz-cd}

\usepackage[T1]{fontenc} 
\usepackage[utf8]{inputenc} 
\usepackage[english]{babel}
\usepackage{amsmath, tensor}
\usepackage{amssymb}
\usepackage{amsthm}
\usepackage{comment}
\usepackage{enumitem}
\usepackage{braket}
\usepackage{graphicx}
\usepackage{sidecap}
\usepackage{hyperref}
\usepackage{tikz-cd}
\usepackage{mathtools}

\definecolor{gold}{rgb}{0.85,0.65,0}

\usepackage[all]{xy}
\usepackage{mathtools}
\usepackage{amsmath,amssymb,tikz-cd,tikz}

\usepackage{thmtools}

\newtheoremstyle{bolditalic}
  {\topsep}{\topsep}%
  {\itshape}%
  {}%
  {\bfseries}%
  {.}%
  { }%
  {\thmname{#1}\ \thmnumber{#2}\thmnote{ (#3)}}

\newtheoremstyle{boldroman}
  {\topsep}{\topsep}%
  {\normalfont}%
  {}%
  {\bfseries}%
  {.}%
  { }%
  {\thmname{#1}\ \thmnumber{#2}\thmnote{ (#3)}}

\declaretheoremstyle[
  headpunct={)},
  headformat={\NAME\ \NUMBER\ (\NOTE)},
]{parenthese}

\theoremstyle{bolditalic}
\newtheorem{thm}{Theorem}[section]
\newtheorem{lem}[thm]{Lemma}
\newtheorem{prop}[thm]{Proposition}
\newtheorem{cor}[thm]{Corollary}

\newtheorem{conj}[thm]{Conjecture}
\newtheorem{ques}[thm]{Question}

\theoremstyle{boldroman}

\newtheorem{rmk}[thm]{Remark}

\numberwithin{equation}{section}

\newcommand{\sing}{\operatorname{sing}}

\newcommand{\im}{\operatorname{im}}

\newcommand{\D}{{\rm D}}

\newcommand{\Aut}{{\rm Aut}}
\newcommand{\Sp}{{\rm Sp}}

\newcommand{\NS}{{\rm NS}}
\newcommand{\Pic}{{\rm Pic}}

\newcommand{\rk}{{\rm rk}}

\newcommand{\End}{{\rm End}}
\newcommand{\Hom}{{\rm Hom}}
\renewcommand{\H}{{\rm H}}
\newcommand{\h}{{\rm h}}
\newcommand{\pt}{{\rm pt}}
\newcommand{\Spec}{{\rm Spec}}

\newcommand{\ext}{{\rm dimExt}}

\newcommand{\Stab}{{\rm Stab}}
\newcommand{\fix}{{\rm Fix}}
\newcommand{\iso}{\cong}

\newcommand{\id}{{\rm id}}

\newcommand{\mono}{\hookrightarrow}
\newcommand{\epi}{\twoheadrightarrow}

\newcommand{\bir}[1][1.75em]{\mathrel{\tikz[baseline]{\draw[dash pattern=on .25em off .1 em,->](0,.58ex)--(#1,.58ex)}}}

\newcommand{\Ext}{{\rm Ext}}

\newcommand{\bl}{{\rm Bl}}
\newcommand{\gr}{{\rm Gr}}
\newcommand{\lgr}{{\rm LGr}}
\newcommand{\pf}{{\rm Pf}}

\renewcommand{\im}{{\rm Im}}
\renewcommand{\ker}{{\rm Ker}}

\newcommand{\mov}{{\rm Mov}}
\newcommand{\reg}{{\rm reg}}
\newcommand{\red}{{\rm red}}

\newcommand{\ml}{{M_{\rm last}}}
\newcommand{\Sigmal}{\Sigma_{\mathrm{last}}}
\newcommand{\mb}{\overline{M}}
\newcommand{\defo}{{\rm Def}}

\newcommand{\cal}{\mathcal}

\newcommand{\kf}{{\cal F}}
\newcommand{\kg}{{\cal G}}

\newcommand{\kh}{{\cal H}}
\newcommand{\ki}{{\cal I}}

\newcommand{\klb}{{\bar{\cal L}}}
\newcommand{\kn}{{\cal N}}

\newcommand{\ko}{{\cal O}}
\newcommand{\kj}{{\cal J}}

\newcommand{\ks}{{\cal S}}

\newcommand{\kx}{{\cal X}}
\newcommand{\ky}{{\cal Y}}
\newcommand{\Extsh}{{\cal Ext}}

\newcommand{\AAA}{\mathbb{A}}
\newcommand{\DDD}{\mathbb{D}}

\newcommand{\MM}{\mathbb{M}}
\newcommand{\MMb}{\overline{\mathbb{M}}}
\newcommand{\ZZ}{\mathbb{Z}}
\newcommand{\QQ}{\mathbb{Q}}
\newcommand{\RR}{\mathbb{R}}
\newcommand{\CC}{\mathbb{C}}

\newcommand{\LL}{\mathbb{L}}

\newcommand{\PP}{\mathbb{P}}

\newcommand{\ST}{\mathrm{ST}}
\newcommand{\clrhom}{{\mathcal{R}\underline{Hom}}}
\newcommand{\GL}{{\rm GL}}
\newcommand{\SL}{{\rm SL}}

\newcommand{\Mat}{{\rm Mat}}
\newcommand{\Jac}{{\rm Jac}}

\newcommand{\fM}{\mathfrak{M}}
\newcommand{\fX}{\mathfrak{X}}
\newcommand{\Var}{{\rm Var}}

\newcommand{\sigmab}{\bar{\sigma}}
\newcommand{\lambdab}{\bar{\lambda}}
\newcommand{\taub}{\bar{\tau}}
\newcommand{\taul}{\tau_{\mathrm{last}}}
\newcommand{\rhob}{\bar{\rho}}
\newcommand{\tr}{\mathrm{tr}}
\newcommand{\pr}{\mathrm{pr}}
\newcommand{\ch}{\mathrm{ch}}
\newcommand{\Affse}{\mathrm{AffSe}}
\newcommand{\Sym}{\mathrm{Sym}}
\newcommand{\Sigmab}{\overline{\Sigma}}
\newcommand{\Hilb}{\mathrm{Hilb}}

\renewcommand{\cong}{\simeq}

\author[V.~Zuliani]{Vanja Zuliani}
\address{\parbox{0.9\textwidth}{Universit\'e Paris-Saclay, CNRS, Laboratoire de Math\'ematiques d'Orsay\\[1pt]
Rue Michel Magat, B\^at. 307, 91405 Orsay, France
\vspace{1mm}}}
\email{{vanja.zuliani@universite-paris-saclay.fr}}

\title{Hodge numbers of a Fano eightfold of K3 type}

\begin{document}
\begin{abstract}
    We construct an explicit semistable degeneration 
    of a Fano eightfold of index three introduced by Flapan-Macr{\'i}-O'Grady-Saccà.
     We deduce its  Hodge numbers and, in particular, we show that it has Picard rank one.
    The Fano variety is of K3 type and it is defined as a connected component of the fixed locus of a suitable antisymplectic involution on a projective variety that is deformation equivalent to the Hilbert scheme of eight points on a K3 surface.
    We also obtain a description of a projective model of the Hilbert square of a K3 surface of genus eight in terms of secant lines to the surface.
\end{abstract}
\maketitle
\tableofcontents

\section{Introduction}
Fano and hyperk{\"a}hler varieties interact in many interesting examples; indeed, several known locally complete families of projective hyperk{\"a}hler varieties have been constructed from  Fano varieties.
The most classical examples are due to Beauville and Donagi who proved that the variety of lines in the cubic 4fold is a hyperk{\"a}hler 4fold and to  Lehn-Lehn-Sorger-van Straten who showed that the variety of twisted cubics in the cubic 4fold yields a  hyperk{\"a}hler 8fold.
Recently Flapan-Macrì-O'Grady-Saccà proposed a converse construction: starting from a hyperk{\"a}hler variety, they produce a Fano variety.
Consider a polarized smooth $8n$fold $(X_n,h_n)$ deformation equivalent to the Hilbert scheme $\Hilb^{4n}(S)$ of $4n$ points on a K3 surface $S$, we assume moreover that the polarization $h_n\in \H^2(X_n,\ZZ)$ is of square two and divisibility two.
Such a variety $X_n$ has a unique antisymplectic involution $\tau_n\colon X_n\to X_n$ whose action on the second cohomology group satisfies $\H^2(X_n,\QQ)_+=h_n\cdot\QQ$.
In \cite{antisymplI,antisymplII} the authors show that 
the fixed locus of the involution $\fix(\tau_n)=F_n\sqcup \Omega_n$ has two connected components, where $F_n$ is a Fano $4n$fold of index three.

The goal of this paper is to study the geometry of the Fano 8fold $F:=F_2$. 
We construct an explicit semistable degeneration having the Fano 8fold as smooth fibre and deduce its Hodge numbers, in particular we show that it has Picard rank one.
To describe the degeneration we study explicit birational modifications and the singularities of Bridgeland moduli spaces.
It would be interesting to understand if $F$ can be realized as zero locus of a section of a vector bundle on a homogeneous variety, if it is $K$-stable. Other open questions are listed at the end of the introduction.

\subsection{Motivations}
The study of  $F$ is motivated by questions in hyperk{\"a}hler geometry, particularly the study of polarized HK 4folds and lagrangian subvarieties, and by the study of higher dimensional Fano varieties.

\subsubsection{Examples of Fano varieties.}
Fano varieties are bounded in every dimension hence a complete classification up to deformation equivalence is, in principle, possible.
In dimension 2 and 3 the classification is complete.
Most of the examples can be constructed as degeneracy loci of morphisms of vector bundles on homogeneous or toric varieties.
To the knowledge of the author the only examples that are not described this way are 
\begin{itemize}
	\item moduli spaces of vector bundles on curves studied in \cite{drezet_narasimhan,desale_ramanan},
	\item the Fano varieties $F_n$ described above.
\end{itemize}
Let us also note that the varieties from the first group have even index and the varieties $F_n$ have  index three.
\begin{ques}
	Is the variety $F_2$ a zero section of a vector bundle on some homogeneous variety ?
\end{ques}

\subsubsection{Hyperk{\"a}hler 4fold}
The problem of unirationality of moduli spaces of polarized hyperk{\"a}hler 4folds is, in some cases, open.
As we explain below, using the results in \cite{antisymplI,antisymplII}, we can  associate a Fano variety to any polarized
hyperkähler 4fold. If the moduli space of such Fano varieties is unirational then this would suggest that the
moduli space of polarized hyperkähler is unirational as well.
By lattice computations there are two orthogonal vectors in the extended K3 lattice $\lambda,h_n\in \tilde{\Lambda}:=U_1\oplus U_2\oplus U^{\oplus 2}\oplus E_8(-1)^{\oplus 2}$, both of divisibility two and with the following squares $\lambda^2=2,h_n^2=8n-2, n\geq 1$.
Here $U, U_1, U_2$ are three copies of the hyperbolic plane, $E_8$ is the unique positive
definite even unimodular lattice of rank 8 and $I(d)$ is the lattice $\ZZ$ whose square of a generator is $d$, see \cite{Deb_Mac_HK} for details.
We will denote by $e_i,f_i$ a standard basis of $U_i, i=1,2$ hence we can choose $\lambda:=e_2+f_2$ and $h_n:=e_1+(4n-1)f_1$.
With this notation 
\begin{equation*}
	h_n^{\perp}=I(-2(4n-1))\oplus U_2\oplus U^{\oplus 2}\oplus E_8(-1)^{\oplus 2}= \Lambda_{K3^{[4n]}}\iso I(-2(4n-1))\oplus\Lambda_{K3}
\end{equation*}
and $\lambda\in U_2\oplus U^{\oplus 2}\oplus E_8(-1)^{\oplus 2}\iso \Lambda_{K3}$ is the lattice of a K3 surface.
Similarly
\begin{equation*}
	\lambda^{\perp}=I(-2)\oplus U_1\oplus U^{\oplus 2}\oplus E_8(-1)^{\oplus 2}= \Lambda_{K3^{[2]}}\iso I(-2)\oplus\Lambda_{K3}
\end{equation*}
and $h_n\in U_1\oplus U^{\oplus 2}\oplus E_8(-1)^{\oplus 2}\iso \Lambda_{K3}$ is the lattice of a K3 surface.
For a projective K3 surface $S$ and  a generic stability condition $\sigma\in \Stab(\D(S))$, the  moduli space $M_{\sigma}(\lambda)$ of $\sigma$-semistable objects in $\D(S)$ is a projective HK manifold of $K3^{[2]}$ deformation type.
Moreover $h_n\in \lambda^{\perp}=\H^2(M_{\sigma}(\lambda),\ZZ)$
is a polarization.
Similarly $X_n:=M_{\sigma}(h_n)$ is a projective HK manifold of $K3^{[4n]}$ deformation type with a polarization $\lambda$ of square 2 and divisibility 2.
By Global Torelli theorem there is a unique antisymplectic involution $\tau_n$ on $X_n$ such that $\H^2(X_n,\QQ)_{+}=\QQ\lambda$.
In \cite{antisymplI,antisymplII} the authors show that the fixed locus of the involution $\fix(\tau_n)=F_n\sqcup \Omega_n$ has two connected components, where $F_n$ is a Fano 8fold of index three.
Hence the correspondence associates to each $M_{\sigma}(v)$ the Fano variety $F_n\subset X_n$.

Let us summarise some known result for $n=1,2,3$.
\begin{center}
    \begin{tabular}{c|c ccc}
        $n$ & $(K3^{[2]}-\textrm{type},h_n)$         & $X_n$& $F_n$ & $\Omega_n$  \\
        \hline
         1& lines in a cubic 4fold $Y$& LLSvS &  $Y$ & general type\\
         2&K{\"u}chle? 
         &&we compute the $h^{p,q}$&\\
         3& Debarre-Voisin &&&\\
    \end{tabular}
\end{center}
The general HK of $K3^{[2]}$ type and polarization $h_1$ (resp. $h_3$) is constructed as the Fano variety of lines in a cubic 4fold $Y$ (resp. is the Debarre-Voisin construction, see \cite{DV}).
This shows that the moduli spaces of such varieties are unirational.
By the work of \cite{LLSvS,Deb_Mac_HK} any HK of $K3^{[4]}$ type with a polarization of square 2 and divisibility 2 is an LLSvS and the Fano component of the fixed locus is a cubic 4fold $Y$.
Moreover by 
\cite{antisymplI,antisymplII} the component $\Omega_1$ is a general type 4fold.
Alexander Kuznetsov asks whether a general $K3^{[2]}$ with polarization of square 14 and divisibility 2 can be constructed from considering the moduli space of twisted cubic curves in a K{\"u}chle 4fold of type C5, see \cite{Kuznetsov_kuhle}.

\subsubsection{Lagrangian subvarieties}
After the construction of the OG10 as a moduli space of sheaves on a $K3^{[2]}$, see \cite{alessio1,alessio2}, it become clear that the  study of Lagrangian subvarieties of HK manifolds is an essential step in the construction of new HK varieties.
The heuristic is that a family of Lagrangian subvarieties that covers a HK $X$ should give rise to a symplectic moduli space of vector bundles on $X$.
In \cite{Ogrady_talk} a \textit{covering family of Lagrangian subvarieties} is defined and many interesting conjectures are discussed, one of them is the following
\begin{conj}[O'Grady]
    There is a positive integer $a\geq 1$ and a  family of Lagrangian subvarieties $\pi:\mathcal{Z}_S\to X_n$ over a base $S$ such that
    \begin{itemize}
        \item $\mathcal{Z}_S$ covers $X_n$, i.e., $\pi$ is surjective,
        \item for all $s\in S$, $\mathcal{Z}_s=\sum_i m_i \Gamma_i\subset X_n$ is an effective cycle with $\Gamma_i\subset X_n$ Lagrangian subvariety,
        \item there is a $s_0\in S$ such that $\mathcal{Z}_{s_0}=a \Omega_n$.
    \end{itemize}
\end{conj}
We expect that the analogue question for $F_n$ have negative answer.

\subsection{Results}
In light of the above  discussion,  it is natural to expect that the Fano varieties $F_n$ will play a role similar to that of the cubic 4fold $F_1$. 
The first natural question is about the Hodge structure of such varieties.
We study the Fano 8fold  $F:=F_2$ of index 3   and construct a semistable degeneration, i.e., a projective flat family $\kf\to \DDD$ over the disc with regular total space $\kf$, the fibre $\kf_t$ is smooth for $t\neq 0$, the scheme $\kf_0$ is reduced and has two smooth irreducible components that meet transversely.
\begin{thm}\label{thm_main_intro}
    There exists a semistable degeneration over the disc $\kf\to \DDD$ such that 
    \begin{itemize}
        \item for $t\neq 0$ the fibre $\kf_t$ is deformation equivalent to $F$,
        \item the central fibre is a transverse union $\kf_0=\Sigma\cup (Q_4/Y)$ and $\Sigma\cap (Q_4/Y)=Q_3/Y$, 
    \end{itemize}
     where 
      $Q_4/Y$ (resp. $Q_3/Y$) is a fibration in smooth quadric 4folds (resp. 3folds) over a smooth cubic 4fold $Y$ and
         $\Sigma$ is a flop of the blowup $\bl_S\PP^8$ of  $\PP^8$ in a K3 surface $S\subset \PP^8$ of genus eight.
\end{thm}
The flop of $\bl_S\PP^8$ is explicit
\begin{equation*}
        \xymatrix{
    \bl_S\PP^8\ar[dr]&&\Sigma\ar[dl]\\
    &\overline{X}
    }
    \end{equation*}
where $\bl_S\PP^8\to \overline{X}$ is a contraction of a $\PP^1$ bundle over $S^{[2]}\subset \overline{X}$ and $\Sigma\to \overline{X}$ is a contraction of a $\PP^2$ bundle over $S^{[2]}$.
In Section \ref{section_CS} we recall some results from \cite{morrison_clemens_schmid} and prove that the monodromy of the family $\kf\to \DDD$ is zero.
This implies that the motivic nearby cycle
    \[
    \psi_f^{\mathrm{mot}}:=[\Sigma]+[Q_4/Y]-(1+\LL)[Q_3/Y]
    \]
gives the Hodge numbers of $\kf_t,t\neq 0$.
\begin{cor}
The Hodge diamond of $F$ is 
      \begin{equation*}
    \begin{tikzcd}[sep=tiny]
        \H^8&0&0&1&22&253&22&1&0&0\\
        \H^6&&0&0&1&22&1&0&0&\\
        \H^4&&&0&1&22&1&0&&\\
        \H^2&&&&0&1&0&&&\\
        \H^0&&&&&1&&&&\\
    \end{tikzcd}
\end{equation*}
    and the odd cohomology of $F$ is zero.
\end{cor}

\begin{proof}[Sketch of the proof of Theorem~\ref{thm_main_intro}.]
In the first part we follow \cite{antisymplI}: consider the Beauville-Mukai system, that is a moduli space $M_{-1}$ of Gieseker stable sheaves on a K3 surface $S$ of genus 8 that with a Lagrangian fibration $\pi\colon M_{-1}\to {\PP^8}^{\vee}$ and an antisymplectic involution $\tau_{-1}$.
 Moreover the zero section section ${\PP^8}^{\vee}\subset M_{-1}$ of the Lagrangian fibration is one component of the fixed locus $\fix(\tau_{-1})$. 
We observe that the fixed part by $\tau_{-1}$ in the cohomology $\H^2(M_{-1},\QQ)_{+}$ has dimension two, hence we want to contract a divisor whose class in $\H^2(M_{-1},\QQ)$ is invariant by $\tau_{-1}$.
To achieve this we perform a sequence of Mukai flops $c_i,c_i',i=-1,0,1$ and a divisorial contraction $g$  given by wall crossing with respect to Bridgeland stability.
\begin{equation*}
    \xymatrix{
M_{-1}\ar[d]^{\pi}\ar[dr]^{c_{-1}}&&M_{0}\ar[dl]^{c_{-1}'}\ar[dr]^{c_{0}}&&M_{1}\ar[dl]^{c_0'}\ar[dr]^{c_1}&&M_2=M_{\rm last}\ar[dl]^{c_1'}\ar[d]^{g}\\
{\PP^8}^\vee&\overline{M}_{-1}&&\overline{M}_{0}&&\overline{M}_{1}&\overline{M}.
    }
\end{equation*}
All the moduli spaces have antisymplectic involutions and we have an induced sequence of flops and divisorial contractions for the connected component ${\PP^8}^{\vee}\subset M_{-1}$ of $\fix(\tau_{-1})$.
\begin{equation*}
    \xymatrix{
{\PP^8}^\vee=\Sigma_{-1}\ar[d]^{\pi_|}\ar[dr]&&\PP^8=\Sigma_{0}\ar[dl]\ar[dr]&&\bl_S\PP^8=\Sigma_1\ar[dl]^{{c'_0}_|}\ar[dr]^{{c_1}_|}&&\Sigma_{\rm last}\ar[dl]^{{c'_1}_|}\ar[d]^{g_|}\\
{\PP^8}^{\vee}&\pt&&\PP^8&&\overline{\Sigma}_1&\Sigmab
    }
\end{equation*}
In Section~\ref{section_flops} we continue the work of \cite{antisymplI,arav24} and describe the flops of $M_i,\Sigma_i$, in particular we describe the singularities of $\overline{M}_i,i=-1,0,1$ and prove that $\Sigmab_1$ is normal, see Proposition~\ref{prop_barsigma1_normal}.
In Section~\ref{Section_contraction} we study the divisorial contraction $g$ restricted to $\Sigmal$ and compute the analytic neighborhood of the singularities of $\Sigmab$.
Moreover we prove that $g_|\colon \Sigmal=\bl_Y\Sigmab\to \Sigmab$ is the blowup of the singular locus which is a smooth cubic 4fold $Y\subset \Sigmab$ and the exceptional divisor $E\to Y$ is a fibration in Lagrangian Grassmannians $E=\lgr(4)/Y$, see Propositions~\ref{prop_barsigm_normal},~\ref{prop_blowup_sigma}.
In Section~\ref{section_semistable} we recall some results from \cite{namikawa,markman10a,antisymplI} and we study the deformations of $\Sigmab$: this  gives us a flat family over the disc
$\ky'\to \DDD$ 
where $\ky'_t,t\neq 0$ is deformation equivalent to the Fano eightfold $F$, the central fibre is $\ky'_0=\Sigmab$, the singular loci coincides $\ky'_{\sing}=\Sigmab_{\sing}$ and the singularity of $\ky'$ are locally analytically of the form $\CC^4\times \{q=0\}$ where $q\in \CC[x_0,\dots,x_4,t]$ is a quadric of maximal rank, see Proposition~\ref{prop_family_Y'}.
The blowup in the singular locus $\kf:=\bl_{\ky'_{\sing}}\ky'\to \DDD$ is the desired family, see Proposition~\ref{prop_bl_fundamental_family}.
\end{proof}

We conclude the introduction with an example that arises naturally in the study of $F$ and it gives an interesting description of a projective model of the Hilbert scheme of a K3 surface of genus eight in terms of the secant lines of the surface.
Let $V_6$ be a $\CC$-vector space of dimension six, let us consider the birational map
\begin{equation*}
    \begin{split}
        \phi\colon \PP(\wedge^2 V_6)&\bir \PP(\wedge^4 V_6)\\
        [\omega]&\mapsto [\omega\wedge\omega].
    \end{split}
\end{equation*}
 For a generic $\PP^8\subset \PP(\wedge^2 V_6)$ the intersection $S:=\gr(2,4)\cap \PP^8$ is a smooth K3 surface of genus eight and by the work of Mukai all the K3 surfaces of genus eight arise this way.
\begin{thm}[{Theorem~\ref{thm_geom_descr_contr}}]
    There is a projective model $S^{[2]}\subset \PP(\wedge^4 V_6)$ obtained as the closure of the images trough $\phi$ of the 2-secant lines to $S$ in $\PP^8$, i.e.,
    \[
    S^{[2]}=\overline{\bigcup_{p,q\in S}\phi(L_{p,q})}
    \]
    where $L_{p,q}$ is the line in $\PP^8$ passing trough $p$ and $q$.
\end{thm}
See also \cite[Theorem 3.7]{ili_ren} for a different proof of the same result.

\subsection{Open questions}
We recall that if $Y$ is a smooth cubic 4fold not containing a plane then the LLSvS eightfold associated to $Y$ has an antisymplectic involution and $Y$ is one of the two connected components of the fixed locus, see \cite{LLSvS,antisymplI,antisymplII}. 
Hence it is interesting to ask if the Fano eightfold $F$ shares some properties with the cubic 4folds, we list here some questions that interest us
\begin{itemize}
    \item do we have a Torelli theorem for $F$?
    \item is $F$ irrational?
    \item is $F$ a $K$-stable variety?
\end{itemize}

The BGMN conjecture proved by Tevelev-Torres in \cite{TT} provides a semiorthogonal decomposition of the derived category $\D(N)$, where $N$ is the moduli space of rank 2 vector bundles of odd degree on a curve $C$ of genus at least 2.
The factors of the semiorthogonal decomposition are equivalent to the derived category of the symmetric powers of $C$.
\begin{conj}[Flapan-Macrì-O'Grady-Saccà]
    The Fano eightfold $F$ obtained as fixed locus in $X$ admits the following semiorthogonal decomposition
    \begin{equation*}
        \D(F)=\langle \D(K),\D,\D,\D,\ko,\ko(1),\ko(2)\rangle
    \end{equation*}
    where $\D$ is a K3 category for which $X=(M_{\sigma}(h),v)$ is a moduli space of objects on $\D$ with a polarization $v$.
    Moreover $K$ is a HK 4fold \textit{strange dual} to $X$, i.e., $K=M_{\sigma}(v)$.
\end{conj}
Note that the Hodge numbers of $F$ agree with this conjecture.
It is natural to attack this conjecture by a degeneration method and by a \textit{simultaneous categorical resolution} defined  by Kuznetsov in \cite{kuz_simult_cat_res}.

\subsection{Acknowledgments}
I would like to thank my advisor  Emanuele Macr{\`i} for suggesting me such a beautiful problem and for the support  during the work and the writing of this article. 
I would like to also thank Pietro Beri, Aaron Bertram, Alessio Bottini, Franco Giovenzana, Alexander Kuznetsov, Kieran O'Grady, {\'A}ngel-David R{\'i}os-Ortiz,  
Francesca Rizzo and Franco Rota for helpful conversations. 
The author was supported by the ERC Synergy Grant 854361 HyperK.

\section{Birational modification of moduli spaces}
In this section we describe explicitly the birational modifications of a moduli space $M$ of sheaves on a K3 surface induced by wall crossing with respect to Bridgeland stability.
Moreover we describe the induced birational modification of the fixed locus of some antisymplectic involution.
In the first part we recall some result from \cite{antisymplI,antisymplII,arav24}.
In the second and third part we give a geometric description of some contractions of moduli spaces and the induced contractions on the fixed loci of involutions.
In particular we compute the line bundles inducing the birational modifications of the fixed loci and give a description of a projective model of a Hilbert square of a K3 surface of genus eight.

\subsection{Mukai flops}\label{section_flops}
We recall \cite[Example 3.25(b)]{antisymplI}.
Let us fix a smooth projective K3 surface $S$ of genus $8$ and $\Pic(S)=H\ZZ$, where $H$ is the ample generator of self intersection $H^2=2\cdot8-2$.
Let us also fix the Mukai vector $v=(0,h,1-8)$, where $h\in \H^2(S,\ZZ)$ is the first Chern class of $H$.
The moduli space $M_h=M_{-1}$ of Gieseker semistable sheaves with pure support  and Mukai vector $v$ has a natural Lagrangian fibration
\begin{equation}
    \pi\colon M_h\to { \PP^8}^\vee
\end{equation}
that associates to each sheaf the class of its fitting support in the linear system $|H|\cong \PP^8$.
The smooth curves $C\in \PP^8$ has as fibre $\pi^{-1}(C)=\Pic^0(C)$.

By a wall crossing of moduli spaces in $\Stab(S)$ we obtain the following diagram
\begin{equation}\label{eq_wallcros}
    \xymatrix{
M_h=M_{-1}\ar[d]^{\pi}\ar[dr]^{c_{-1}}&&M_{0}\ar[dl]^{c_{-1}'}\ar[dr]^{c_{0}}&&M_{1}\ar[dl]^{c_0'}\ar[dr]^{c_1}&&M_2=M_{\rm last}\ar[dl]^{c_1'}\\
{\PP^8}^\vee&\overline{M}_{-1}&&\overline{M}_{0}&&\overline{M}_{1}&
    }
\end{equation}
where $M_{-1},M_0,M_1,\ml$ are moduli spaces of Bridgeland stable objects in $\D(S)$ with Mukai vector $v$, they are smooth projective hyperkähler varieties of dimension $16$.
If we denote by $\DDD(-):=\clrhom_S(-,\ko_S)[1]$ the shifted dual, the autoequivalence $\Psi:=\ko_S(-H)\otimes\DDD$ induces antisymplectic involutions $\tau_i$ on $M_i$ for $i=-1,0,1,2$.
The morphism 
\begin{equation}\label{eq_mukaiflop_i}
\xymatrix{
M_{i}\ar[dr]^{c_i}&&M_{i+1}\ar[dl]^{c'_i}\\
&\overline{M}_i&
}
\end{equation}
is a standard Mukai flop described as follows.
\begin{lem}[{\cite[Example 3.21]{antisymplI}}]
    The Mukai flop~\eqref{eq_mukaiflop_i} for  $i=-1$ has as exceptional locus $\Sigma_{-1}$ in $M_{-1}$ the zero section of the fibration $\pi:M_{-1}\to \PP^8$, moreover 
    \[\Sigma_{-1}=\PP(\Ext^1(\ko_S(-H)[1],\ko_S[1])=\PP(\H^0(S,\ko_S(H))).
    \]
    The exceptional locus $\Sigma_0$ of $c'_{-1}$ in $M_0$ is given by $\PP(\Ext^1(\ko_S,\ko_S(-H)[2]))=\PP(\H^0(S,\ko_S(H))^{\vee})$.
\end{lem}
\begin{lem}[{\cite[Example 3.25]{antisymplI}}]
    The Mukai flop~\eqref{eq_mukaiflop_i} for  $i=0$ has as exceptional locus the $\PP^6$-bundle in the Zariski topology over the product $S\times S\subset \overline{M}_0$ and the fibre over $(p,q)$ is 
    \[
    \PP(\Ext^1(\DDD(\kj_p)\otimes\ko_S(-H),\kj_p))\subset \overline{M}_0.
    \]
    The exceptional locus of $c'_0$ in $M_1$ is the dual bundle.
\end{lem}
\begin{lem}[{\cite[Example 3.25]{antisymplI}}]\label{lem_c1}
    The Mukai flop~\eqref{eq_mukaiflop_i} for  $i=1$ has as exceptional locus the $\PP^4$-bundle in the Zariski topology over the product $S^{[2]}\times S^{[2]}\subset \overline{M}_1$ and the fibre over $(\kj,\kj')\in S^{[2]}\times S^{[2]}$ is 
    \[
    \PP(\Ext^1(\DDD(\kj)\otimes\ko_S(-H),\kj'))\subset \overline{M}_1.
    \]
    The exceptional locus in $M_2$ is the dual bundle.
\end{lem}

We remark that the antisymplectic involutions $\tau_i$ preserve the exceptional loci and descend to involutions $\bar{\tau}_i$ on $\overline{M}_i,i=-1,0,1$, see \cite[Lemma 3.24]{antisymplI}.
In particular $\Sigma_i$ is a component of $\fix(\tau_i)$ for $i=-1,0$.
Moreover by \cite[Lemmata 5.4, 5.9]{antisymplI} $\Sigma_0$ is not contained in the exceptional locus of $c_0$ and its strict transform $\Sigma_1\subset M_1$ is not contained in the exceptional locus of $c_1$.
We will use the notation $\Sigmal=\Sigma_2\subset M_2$ for the strict transform of $\Sigma_1$.
Hence $\Sigma_i,i=-1,0,1,2$ are all birational and they are connected components of the respective $\fix(\tau_i)$.
By the discussion above and \cite[Proposition 7.5]{arav24} the diagram~\eqref{eq_wallcros} restricts to
\begin{equation}\label{eq_diag_fixedloc}
    \xymatrix{
{\PP^8}^\vee=\Sigma_{-1}\ar[dr]&&\PP^8=\Sigma_{0}\ar[dl]\ar[dr]&&\bl_S\PP^8=\Sigma_1\ar[dl]^{{c'_0}_|}\ar[dr]^{{c_1}_|}&&\Sigma_{\rm last}\ar[dl]^{{c'_1}_|}\\
&\pt&&\PP^8&&\overline{\Sigma}_1&
    }
\end{equation}
where ${c'_0}_|$ is the blowup of $S\subset \PP^8=\PP(\H^0(S,\ko_S(H))^{\vee})$.

\subsection{The contraction ${c_1}_{|}$.}
The goal of this section is to prove that ${c_1}_|$ is an algebraic fibre space and compute the line bundle inducing it, see Proposition~\ref{prop_c1}. Along the way we study the singularities of $\overline{M}_1$ and prove that $\overline{\Sigma}_1$ is normal, see Proposition~\ref{prop_barsigma1_normal}.
Let us recall that the contraction $c_1\colon M_1\to \overline{M}_1$ contracts a $\PP^4$ bundle over $S^{[2]}\times S^{[2]}$. 
The involution $\taub_1$ on $\overline{M}_1$ restricts to $S^{[2]}\times S^{[2]}$ interchanging the two factors, hence the fixed locus is the diagonal $S^{[2]}\subset S^{[2]}\times S^{[2]}$.
More precisely for a point $(\kj,\kj)\in S^{[2]}$ the fibre is $c_1^{-1}((\kj,\kj))=\PP(\Ext^1(\DDD(\kj)\otimes\ko_S(-H),\kj))$ and the action of $\tau_1$ on it is the linear involution given by $\Psi$.
As computed in \cite[Propositions 8.4,7.2]{arav24} we have that 
\[
\Sigma_1\cap c_1^{-1}(\kj,\kj)=\PP\left({\Ext^1(\DDD(\kj)\otimes\ko_S(-H),\kj)}^-\right)\iso \PP^1
\]
and by \cite[Proposition 5.2]{antisymplI} we have  
\[
\Sigmal\cap {c'}_1^{-1}(\kj,\kj)=\PP\left({\Ext^1(\DDD(\kj)\otimes\ko_S(-H),\kj)}^{\vee,+}\right)\iso \PP^2.
\]
\begin{cor}[{\cite[Propositions 8.4,7.2]{arav24},\cite[Proposition 5.2]{antisymplI}}]
    The flop
    \begin{equation}\label{eq_sigma_flop}
        \xymatrix{
    \Sigma_1\ar[dr]&&\Sigmal\ar[dl]\\
    &\overline{\Sigma}_1
    }
    \end{equation}
    contracts a $\PP^1$-bundle over $S^{[2]}\subset \overline{\Sigma}_1$ which is replaced by a $\PP^2$-bundle in $\Sigmal$. 
\end{cor}
\begin{prop}\label{prop_barsigma1_normal}
    The image $\overline{\Sigma}_1$ of ${c_1}_|\colon \Sigma_1\to \overline{M}_1$ is normal.
\end{prop}
\begin{proof}
We observe that 
$ \overline{\Sigma}_1\cap \overline{M}_{1,\reg}$ is smooth since it is the fixed  locus of an involution on $\overline{M}_1$.
    When $y\in \overline{M}_1\setminus \overline{M}_{1,\reg}=S^{[2]}\times S^{[2]}$ we use the description of the singular locus of $\overline{M}_1$ given in \cite{arbarellosaccaunupdate}.
    By the modular description of the contraction $c_1$, $\overline{M}_1$ is a moduli space of Bridgeland semistable objects.
    The singular locus of the base $\overline{M}_1$ consists of points $y=[E\bigoplus \Psi(E)]$ where $E,F:=\Psi(E)\in \D(S)$ are non isomorphic stable objects, with respect to a stability condition on a wall, see \cite[Section 3.4]{antisymplI}.
    By \cite{arbarellosaccaunupdate}, locally analytically around the point $y$, the base is isomorphic to the germ around zero of
    \[
    \Ext^1(E,E)\oplus \Ext^1(F,F)\oplus\left(\mu^{-1}(0)//\CC^*\right)
    \]
    where 
    \[
    \begin{split}
        \mu \colon &\Ext^1(E,F)\oplus\Ext^1(E,F)^{\vee}\to \CC\\
        &e\oplus e^{\vee}\mapsto e^{\vee}(e)
    \end{split}
    \]
    and $\CC^*$ acts as follows $t.e=te, t.e^{\vee}=t^{-1}e^{\vee},t\in \CC^*$.
    It follows that $\mu^{-1}(0)//\CC^*\subset \Affse$ is contained  in the affine cone over the Segre embedding of $\PP(\Ext^1(E,F))\times\PP(\Ext^1(E,F)^{\vee})\subset \PP^{\ext^1(E,F)^2-1}$, we will use the coordinates $u_{ij}$, hence the image of the Segre embedding is defined by the two by two minors of the matrix with entries $u_{ij}$.
    The quadratic form descend to a linear form $\mu^l=\sum_i u_{ii}$ on $\mathbb{A}^{\ext^1(E,F)^2}$ hence $\mu^{-1}(0)//\CC^*=\Affse\cap\{\mu^l=0\}$.
    The involution $\taub_1$ is induced by the functor $\Psi$ and restricts to involutions on $\Ext^1(E,E)\oplus \Ext^1(F,F)$ and $\mu^{-1}(0)//\CC^*$. 
    Hence $\fix(\taub_1)$ can be written, in a neighborhood of $y$, as a product.
    The first factor is clearly smooth, the second is
    $\fix(\mu^{-1}(0)//\CC^*,\taub_{11})$ where $\taub_{11}$ is induced by $\Psi\colon\Ext^1(E,F)\to \Ext^{1}(\Psi(F),\Psi(E))\iso \Ext^1(E,F)$ and by the dual action on $\Ext^1(F,E)^{\vee}$, see \cite[Lemma 3.24]{antisymplI}.
    Moreover we denote by $x_i,y_j\in \Ext^1(E,F)$ a bases such that $\tau_1(x_i)=x_i, \tau_1(y_j)=-y_j$.
    We observe that the preimage of $\fix(\mu^{-1}(0)//\CC^*,\taub_{11})$ trough the quotient morphism
    \[
    \mu^{-1}(0)\to \mu^{-1}(0)//\CC^*\subset \Affse
    \]
has two irreducible components
\[
\begin{split}
    U_1:=&\{x_i=0\}\cap \left\{\sum_j y_j^{\vee}(y_j)=0\right\}\\
    U_2:=&\{y_j=0\}\cap\left\{\sum_i x_i^{\vee}(x_i)=0\right\}.
\end{split}
\]
Both of then are normal closed and $\CC^*$-invariant.
Hence their quotients are normal too.
We know that  $\Sigma_1$ is irreducible hence its image coincides with one of the irreducible components just described.
\end{proof}

\begin{cor}\label{cor_c1rest_isalgfibrespace}
    The morphisms in~\eqref{eq_sigma_flop} are algebraic fibre spaces.
    In particular if $W$ is the line bundle inducing the morphism $c_1\colon M_1\to \overline{M}_1$ then $W_{|\Sigma_1}$ induces ${c_1}_|\colon \Sigma_1\to \overline{\Sigma}_1$. The analogous statement is true for $c'_1$.
\end{cor}
\begin{proof}
    For the first claim it is enough to observe that the base is normal by Proposition~\ref{prop_barsigma1_normal} and the fibres are connected by Lemma~\ref{lem_c1}, see \cite[p. 125]{lazarsfeld_positivityI}.
    The rest follows from \cite[Lemma 2.1.13]{lazarsfeld_positivityI}.
\end{proof}

We now compute the line bundle inducing $c_1\colon M_1\to \overline{M}_1$.
Let us recall that by \cite[Lemma 3.18]{antisymplI} all moduli spaces $M_{i},i=-1,0,1,2$ have the same movable cone 
\[
\mov(M_i)=\RR_{\geq 0}f+\RR_{\geq 0}\lambda
\]
where $f:=\pi^*\ko_{{\PP^8}^{\vee}}(1)$ and $\lambda:=\vartheta(2,-h,3)$ where $\vartheta\colon v^{\perp}\to\H^{2}(M,\ZZ)$ is the Mukai isomorphism, see \cite{kieran95}.
By \cite[Section 3.4]{antisymplI} the contractions 
\begin{itemize}
    \item $c_{-1},c_{-1}'$ are given by $\lambda+5f$,
    \item $c_0,c_0'$ are given by $\lambda+3f$,
    \item $c_1,c_1'$ are given by $\lambda+f$.
\end{itemize}

We consider the blowup ${{c'}_0}_|\colon\bl_S(\PP^8)=\Sigma_1\to \PP^8$ and we denote by $L,E\in \Pic(\Sigma_1)$ the pullback of $\ko_{\PP^8}(1)$ and the exceptional divisor respectively.

\begin{lem}[{\cite[Lemma 8.3]{arav24}}]\label{lem_rest_lamb,f_0}
    The restrictions of $\lambda,f$ on ${\PP^8}^{\vee}$ are
    \[
    \begin{split}
        f_{|{\PP^8}^{\vee}}&=\ko_{{\PP^8}^{\vee}}(1)\\
        \lambda_{|{\PP^8}^{\vee}}&=\ko_{{\PP^8}^{\vee}}(-5),
    \end{split}
    \]
    the restrictions to $\PP^8$ are
    \[
    \begin{split}
        f_{|\PP^8}&=\ko_{\PP^8}(-1)\\
        \lambda_{|\PP^8}&=\ko_{\PP^8}(5).
    \end{split}
    \]
The restrictions of $f,\lambda$ to $\Sigma_1$ are
    \begin{equation*}
        \begin{split}
        f_{|\Sigma_1}&=-L+E\\
        \lambda_{|\Sigma_1}&=5L-3E.
        \end{split}
    \end{equation*}
\end{lem}

We have now proved the following
\begin{prop}\label{prop_c1}
    The morphism ${c_1}_|\colon \Sigma_1\to \overline{\Sigma}_1$ is an algebraic fibre space induced by a multiple of $L-2E$. 
    Moreover it is the contraction of a  $\PP^1$-bundle over $S^{[2]}\mono \overline{\Sigma}_1$. 
\end{prop}

\subsection{Geometric description of ${c_1}_|$}
In this section we give a geometric description of the contraction $c_1$.
Let us fix a vector space $V_6$ of dimension $6$ and consider the Pl\"ucker embedding
\begin{equation*}
    \gr(2,V_6)\mono\PP(\wedge^2V_6).
\end{equation*}
A linear subspace  $V_9\subset \wedge^2 V_6$ of dimension nine is given by six equations $\eta_i,i=1,\dots, 6$ and the orthogonal is $V_9^{\perp}=\CC\langle \eta_1,\dots,\eta_6\rangle\subset \wedge^2 V_6^{\vee}$, of dimension $6$.
We will denote by $\pf(V_6)\subset\PP(\wedge^2 V_6)$ the Pfaffian variety of two forms that does not have maximal rank, it is a hypersurface of degree $3$.

\begin{prop}[{\cite{BD}}]
    For a generic $V_9\subset \wedge^2V_6$
    the surface 
    \[
    S:=\gr(2,V_6)\cap\PP(V_9)
    \]
    is a K3 surface of genus 8.
    The Fano variety of lines of the cubic 4fold $Y_3=\PP(V_9^{\perp})\cap \pf(V_6^{\vee})$ is isomorphic to $S^{[2]}$.
\end{prop}
Now consider the birational map
\begin{equation}
\begin{split}
    \PP(\wedge^2V_6)&\overset{\phi}{\bir} \PP(\wedge^4 V_6)\iso\PP(\wedge^2 V_6^{\vee})\\
    \omega&\mapsto \omega\wedge\omega.
\end{split}
\end{equation}
The main result of this section is the following 
\begin{thm}\label{thm_geom_descr_contr}
    There is a commutative diagram
    \begin{equation*}
        \xymatrix{
        \Sigma_1\ar[r]\ar[d]^{{c_1}_|}&\bl_{\gr(2,V_6^{\vee})}\PP(\wedge^2 V_6^{\vee})\ar[d]\\
        \overline{\Sigma}_1\ar[r]&\PP(\wedge^2 V_6^{\vee})
        }
    \end{equation*}
    where the horizontal morphisms are closed embeddings.
    Moreover $S^{[2]}\subset \PP(\wedge^2 V_6^{\vee})$ is the closure of the image by $\phi$ of the two-secant lines to $S\subset \PP(\wedge^2 V_6)$. 
\end{thm}
We will prove this theorem by a detailed analysis of a resolution of the birational map $\phi$.
Let us start with the following observation.
\begin{lem}
    The rational map $\phi$ is given by the linear system 
    \[
    \ki_{\gr(2,V_6)}\otimes\ko(2)
    \]
\end{lem}
\begin{proof}
    By explicit computations we see that $\phi$ is given by $15$ homogeneous polynomials of degree two moreover all of them vanish on the Grassmannian. To conclude it is enough to observe that \[
    \dim\H^0(\PP(\wedge^2 V_6),\ki_{\gr(2,V_6)}(2))=15.\]
\end{proof}
We will use the notation $\PP^{14}:=\PP(\wedge^2 V_6), \bl \PP^{14}:=\bl_{\gr(2,V_6)}\PP(\wedge^2 V_6).$
To resolve $\phi$, we perform the blow up of $\PP^{14}$ in the Grassmannian and obtain the following commutative diagram
\begin{equation}\label{eq_diagram}
    \xymatrix{
    \bl_S\PP^8\ar[d]\ar@{^{(}->}[r] &\bl_{\gr}\PP^{14}\ar[d]^p\ar[dr]^q&\\
    \PP^8\ar@{^{(}->}[r]&\PP^{14}\ar@{-->}[r]^{\phi}&\PP(\wedge^2 V_6^{\vee}).
    }
\end{equation}

We observe that if we denote by $H$ the pullback by the blowup of the ample divisor on $\PP^{14}$ and by $E$ the exceptional divisor we have 
\[
(p^{-1}\ki_{\gr})\otimes p^*\ko(2)\iso\ki_{E}\otimes\ko_{\bl \PP^{14}}(2H)=\ko_{\bl \PP^{14}}(2H-E).
\]

\begin{lem}
    The morphism $q$ is given by the complete linear system of $\ko_{\bl\PP^{14}}(2H-E)$.
\end{lem}
\begin{proof}
We observe that $\phi\circ p$ is defined on $\bl\PP^{14}\setminus E$ where it is given by the sections in $\ko(2H)$ that vanish on the exceptional divisor. 
Hence $\phi\circ p$ is given by some sections of $\ko(2H-E)$, we now show that it is given by all the global sections.
    Let us note that $p_*\ko_{\bl\PP^{14}}=\ko_{\PP^{14}},p_*\ko_E=\ko_{\gr}$, hence we have
    the short exact sequence
    \[
    0\to p_*(\ko(2H-E)) \to \ko(2)\to \ko_{\gr}(2)\to 0.
    \]
    In particular $p_*(\ko(2H-E))=\ki_{\gr}(2)$.
    We conclude by observing that 
    \[
    \H^0(\bl_{\gr}\PP^{14}, \ko(2H-E))=\H^0(\PP^{14}, \ki_{\gr} (2))=\CC^{15}.
    \]
    and the commutativity of \eqref{eq_diagram} is then clear.
\end{proof}
\begin{lem}\label{lem_blP14}
    The morphism $\bl_{\gr(2,V_6)}\PP^{14}\overset{q}{\to}{\PP^{14}}^{\vee}$ is the blowup $\bl_{\gr(2,V_6^{\vee})}{\PP^{14}}^{\vee}\to {\PP^{14}}^{\vee}$ and the exceptional divisor $E'$ is the strict transform $\hat{\pf}$ by $p$ of $\pf\subset \PP^{14}$.
\end{lem}
\begin{proof}
We start by describing the blow up $p$. Let us fix a point $\omega\in \gr$, up to a change of basis we can assume that $\omega=e_5\wedge e_6$, where $\langle e_1,\dots,e_6\rangle=V_6$. 
Moreover we will use the notation $V_4=\langle e_1,\dots,e_4\rangle$.
The normal bundle of the Grassmannian at $\omega$ is the cokernel of the differential map
\[
\begin{split}
    \Hom(\langle e_5,e_6\rangle,V_6/\langle e_5,e_6\rangle)&\to \Hom(\langle\omega\rangle,\wedge^2V_6/\langle\omega\rangle)\\
    \lambda&\mapsto e_5\wedge\lambda(e_6)+\lambda(e_5)\wedge e_6.
\end{split}
\]
Hence we will use the following identification
\[
\kn_{\gr,\omega}=\wedge^2V_6/\langle e_5\wedge V_6+V_6\wedge e_6\rangle.
\]
Moreover the morphism $q_|\colon \PP(\kn_{\gr,\omega})\to\PP(\wedge^2V_4^{\vee})\subset  {\PP^{14}}^{\vee}$ is induced by the following isomorphism of vector spaces
\[
\begin{split}
     \wedge^2V_6/\langle e_5\wedge V_6+V_6\wedge e_6\rangle&\iso \wedge^2V_4^{\vee}\\
     \eta&\mapsto (\eta\wedge\omega)^{\vee}.
\end{split}
\]
The intersection between the strict transform of the Pfaffian and a fibre of the exceptional bundle
\[
\hat{\pf}\cap\PP(\kn_{\gr,\omega})
\]
is the quadric 4fold of two forms that does not have maximal rank and its image in $\PP(\wedge^2V_4^{\vee})$ is defined by the equation $\eta^{\vee}\wedge\eta^{\vee}=0$.
Moreover if $(\eta\wedge\omega)^{\vee}\wedge((\eta\wedge\omega)^{\vee})\neq 0$ then it is proportional to $\omega$.
From this description it follows that 
\[
q_| \colon \bl_{\gr(2,V_6)}\PP^{14}\setminus \hat{\pf} \to {\PP^{14}}^{\vee} 
\]
is an isomorphism on the image and the image is $\PP(\wedge^2V_6^{\vee})\setminus \gr(2,V_6^{\vee})$.

We study now the restriction 
\[
q_| \colon \hat{\pf}\to \gr(2,V_6^{\vee})
\]
and prove that the fibres have a natural identification with the normal bundle of $\gr(2,V_6^{\vee})\subset\PP(\wedge^2V_6^{\vee})$.
Given $\sigma\in \gr(2,V_6^{\vee})$, we can choose a basis of $V_6$ such that $\sigma=e_1\wedge e_2\wedge e_3\wedge e_4=e_5^{\vee}\wedge e_6^{\vee}$
so that
\begin{equation*}
\begin{split}
    q^{-1}(\sigma)=&(q^{-1}(\sigma)\setminus E )\bigcup(q^{-1}(\sigma)\cap E)\\
    =&\left\{ \omega\in \wedge^2V_6 \text{ of rank }4 \text{ s.t. } \omega\wedge\omega=\sigma\right\}\bigcup\\
    &\left\{[\eta] \in \PP(\kn_{\gr,\omega
    }) \text{ s.t. }\eta\wedge\omega=\sigma \right\}.
\end{split}
\end{equation*}
We deduce that $p_| \colon q^{-1}(\sigma)\to \PP^{14}$ is injective and it gives an identification
\begin{equation}\label{eq_AB}
\begin{split}
    p_|\colon q^{-1}(\sigma)\to &\left\{ \omega \text{ of rank }4 \text{ s.t. } \omega\wedge\omega=\sigma\right\}\bigcup\\
    &\left\{\omega \text{ of rank }2 \text{ s.t. there is a two form }\eta\in \kn_{\gr,\omega
    } \text{ with }\eta\wedge\omega=\sigma\right\}.
\end{split} 
\end{equation}
We denote by $A$ and $B$ the two sets in the union~\eqref{eq_AB}. We observe that 
\[
p(q^{-1}(\sigma))\cap\gr(2,V_6)=B\cap\gr(2,V_6)=\gr(2,V_4)
\]
where $V_4=\langle e_1,\dots,e_4\rangle$.

We observe that for  $\omega=e_5\wedge e_6\in \gr(2,V_6)$
we have that
\[
q(p^{-1}(\omega))\cap \gr(2,V_6^{\vee})=\PP(\wedge^2V_4^{\vee})\cap\gr(2,V_6^{\vee})=\gr(2,V_4^{\vee}).
\]
This gives us an identification
\begin{equation}\label{eq_BA}
\begin{split}
q_| \colon p^{-1}(\omega)=\PP(\kn_{\omega,\gr(2,V_6)})\to \PP(\wedge^2V_4^{\vee})=&\left\{ \eta\in\wedge^2V_4^{\vee}  \text{ of rank } 4 \text{ and }\eta\wedge\eta=e_1^{\vee}\wedge\cdots\wedge e_4^{\vee} \right\}\bigcup\\
    &\gr(2,V_4^{\vee})
\end{split}
\end{equation}
Comparing~\eqref{eq_AB} and~\eqref{eq_BA} we get the claim.
\end{proof}

From now we will use the notation $X:=\bl_S(\PP^8)\subset \bl_{\gr(2,V_6)}(\PP^{14})$, 
$\overline{X}:=q(X)\subset {\PP^{14}}^{\vee}$ and $f:=q_|\colon X\to \overline{X}$.
\begin{cor}
The pullback by $\phi^{-1}$ of the equations $\eta_1,\dots,\eta_6$ defining $\PP^8\subset\PP^{14}$ defines a complete intersection of six quadrics $\Gamma=\bigcap Q_i$, it has two irreducible components
    \[
\Gamma=\overline{X}\cup \gr(2,V_6^{\vee})
    \]
    and 
    \[
    D:=\overline{X}\cap \gr(2,V_6^{\vee})
    \]
    is a prime Weil divisor of $\overline{X}$ and $\gr(2,V_6^{\vee})$
\end{cor}
\begin{proof}
    The linear subspace $\PP^8$ is defined by six linear independent equations $\eta_i\in \wedge^2V_6^{\vee}$, 
    \[
\Gamma=\bigcap Q_i
    \]
    is the intersection of six quadrics 
    \[
    q_i(\sigma):=\eta_i\wedge\sigma\wedge\sigma.
    \]
It is clear that $\gr(2,V_6^{\vee})\subset \Gamma$.
Straightforward computations show that for a $\omega\in \wedge^2V_6$ of rank $6$ the 
equality
\[
\eta_i(\omega)=0
\]
is equivalent to 
\[
q_i(\omega\wedge\omega)=0.
\]
Hence $\overline{X}=q(X)=\overline{\phi(\PP^8)}$ is the irreducible component $\overline{(\Gamma\setminus\pf(V_6^{\vee}})$.
Since $\Gamma$ is given by six equations all the irreducible components have dimension eight.
To conclude the proof of the first claim we need to show that $\pf(V_6^{\vee})\cap \left(\Gamma\setminus \gr(2,V_6^{\vee})\right)$ has dimension strictly less then eight.
Any $\sigma\in \pf(V_6^{\vee})\cap\left(\Gamma\setminus \gr(2,V_6^{\vee})\right)$ has rank $4$.
By the description of the morphism $q$ in the proof of Lemma~\ref{lem_blP14} we have $\omega=e_5\wedge e_6\in \gr(2,V_6)$ and $\eta\in \PP(\kn_{\gr,\omega})$ of rank four such that $\sigma=(\eta\wedge\omega)^{\vee}$. Let us also note that we can assume $\eta \in \wedge^2V_4$.
We compute 
\[
\eta\wedge\eta=\alpha e_1\wedge e_2 \wedge e_3 \wedge e_4
\]
for a non zero $\alpha\in \CC^*$.
We also compute
\[
\begin{split}
    q_i((\omega\wedge\eta)^{\vee})=&(\omega\wedge\eta)^{\vee}\wedge(\omega\wedge\eta)^{\vee}\wedge\eta_i=\alpha e_1^{\vee}\wedge\cdots\wedge e_4^{\vee}\wedge\eta_i\\
=&\alpha\eta_i^{(56)}e_1^{\vee}\wedge\cdots\wedge e_6^{\vee}
\end{split}
\]
Hence the three conditions $\eta_i(\omega)=0,i=1,\dots,6$, $q_i(\sigma)=0,i=1,\dots,6$,
$\omega\in S$ are equivalent.
Let us define the open set $U\subset \PP(\kn_{|S})$ consisting of forms of rank $4$.
The computations above prove that the map
\[
q_|: U\to \pf(V_6^{\vee})\cap\left(\Gamma\setminus \gr(2,V_6^{\vee})\right)
\]
is an isomorphism and we conclude by observing that the domain has dimension $7$.

We observe that $D:=\overline{X}\cap \gr(2,V_6^{\vee})=q(\bl_S(\PP^8)\cap \hat{\pf})$
hence $D$ is irreducible.
Moreover $\Gamma$ is a complete intersection so it is Cohen-Macaulay and in particular connected in codimension $1$. 
This shows that $D$ has dimension $7$. 
\end{proof}
\begin{cor}
    The morphism
    \[
    f : X\to \overline{X}
    \]
    is the blowup of $\overline{X}$ in $D$. In particular $\overline{X}$ is singular in some points of $D$ and $D$ is not a Cartier divisor.
\end{cor}

We will now compute the divisor on $X$ giving the morphism $f\colon X\to \overline{X}.$

\begin{lem}\label{lem_surj_restrH0}
	The restriction morphism $\H^0(\bl_{\gr(2,V_6)}\PP^{14},2H-E)\to \H^0(\bl_S\PP^8,(2H-E)_|)$ is an isomorphism.
	Moreover $\H^0(\PP^{14},\ki_{\gr(2,V_6)}^k(m))\to \H^0(\PP^8,\ki_S^k(m))$ is surjective for $m\geq 2k\geq 0$, in particular the algebra $\bigoplus_{k\geq 0} \H^0(\bl_S\PP^8,k(2H-E))$ is generated in degree one.
\end{lem}
\begin{proof} 
	Let us consider the short exact sequence
	\[
	0\to \ki_{\bl\PP^8}(2H-E)\to\ko_{\bl\PP^{14}}(2H-E)\to \ko_{\bl\PP^8}(2H-E) \to 0
	\]
	taking the pushforward we get
	\[
	0\to p_*(\ki_{\bl\PP^8}(2H-E))\to\ki_{\gr}(2)\to 
	\ki_{S/\PP^8}(2)
	\to0
	\]
	hence 
	\[
	p_*(\ki_{\bl\PP^8}(2H-E))=\ki_{\gr}\cdot \ki_{\PP^8}(2).
	\]
	We observe that we are left to prove
	\begin{itemize}
		\item $H^0(\PP^{14},\ki_{\gr}\cdot \ki_{\PP^8}(2))=0$,
		\item $H^1(\PP^{14},\ki_{\gr}\cdot \ki_{\PP^8}(2))=0$.
	\end{itemize} 
	
	For the first vaishing let us note that $\PP^8$ and $\gr$ intersects transversely hence $\ki_{\gr}\cdot\ki_{\PP^8}=\ki_{\gr}\cap \ki_{\PP^8}$.
	Let us consider a section $\varphi\in \H^0(\PP^{14},\ki_{\PP^8}\cap\ki_{\gr}(2))$ we want to prove that it is constant equal to zero in a neighborhood of some $p\in \PP^8\cap \gr$. 
	Let us choose analytic coordinates $z_i$ around $p$ such that $z_1,\dots,z_8$ are coordinates for $\PP^8$ and $z_6,\dots,z_{14}$ are coordinates for $\gr$. 
	By assumption $\partial_{z_i}\phi=0,i=1,\dots,14$ in a neighborhood of $p$, hence the claim.
	
	To prove the vanishing $H^1(\PP^{14},\ki_{\gr}\cap \ki_{\PP^8}(2))=0$
	let us consider the exact sequence
	\begin{equation}\label{eq_glsections}
		\begin{split}
			0&\to \H^0(\PP^{14},\ki_{\gr}(2))\to \H^0(\PP^{14},\ko(2))\to \H^0(\PP^{14},\ko_{\gr}(2))\to\\
			&\to \H^1(\PP^{14},\ki_{\gr}(2))\to 0.
		\end{split}
	\end{equation}
	Since we know the dimension of all global sections in~\eqref{eq_glsections}
	we get that $\H^1(\PP^{14},\ki_{\gr}(2))=0.$
	We now consider
	\begin{equation*}
		\begin{split}
			0&\to 0=\H^0(\PP^{14},\ki_{\PP^8}\cap\ki_{\gr}(2))\to \H^0(\PP^{14},\ki_{\gr}(2))\to \H^0(\PP^{8},\ki_{S/\PP^8}(2))\to \\
			&\to \H^1(\PP^{14},\ki_{\PP^8}\cap\ki_{\gr}(2))\to 0
		\end{split}
	\end{equation*}
	and prove that  $\H^0(\PP^{14},\ki_{\gr}(2)), \H^0(\PP^{8},\ki_{S/\PP^8}(2))$ have the same dimension.
	We consider the exact sequence 
	\[
	0\to \ki_{S/\PP^8}(2)\to \ko_{\PP^8}(2)\to \ko_S(2)\to 0
	\]
	and the long exact sequence in cohomology, by \cite[Theorem 6.1]{saint_donat} the map 
	\[
	\H^0(\PP^8,\ko_{\PP^8}(2))\to \H^0(S,\ko_S(2))
	\]
	is surjective, hence we have the exact sequence
	\[
	0\to \H^0(\PP^8,\ki_{S/\PP^8}(2))\to \H^0(\PP^8,\ko(2))\to \H^0(\PP^8,\ko_S(2))\to 0.
	\]
	By Riemann-Roch $\h^0(S,\ko_S(2H))=30$ so we conclude that
	$h^0(\PP^8,\ki_{S/\PP^8}(2))=45-30=15$.
	We also know that 
	$\h^0(\PP^{14},\ki_{\gr}(2))=15$.

	Let us consider the short exact sequence
	\[
	0\to \ki_{\bl\PP^8}(mH-kE)\to\ko_{\bl\PP^{14}}(mH-kE)\to \ko_{\bl\PP^8}(mH-kE) \to 0
	\]
	taking the pushforward we get
	\[
	0\to p_*(\ki_{\bl\PP^8}(mH-kE))\to\ki^k_{\gr}(k)\to 
	\ki_{S/\PP^8}^k(m)
	\to0
	\]
	hence 
	\[
	p_*(\ki_{\bl\PP^8}(mH-kE))=\ki_{\gr}^k\cdot \ki_{\PP^8}(m).
	\]
	We now prove that $\H^1(\PP^{14},\ki_{\gr}^k\cdot \ki_{\PP^8}(m) )=0$ for $m\geq 2k$.
	Let us denote by $R=\CC[x_0,\dots, x_{14}]$ polynomial ring, by $a= (x_0,\dots, x_{14})$ the irrelevant ideal and by $I,J$ the homogeneous ideals of the Grassmannian and the plane $\PP^8$ respectively.
	Since $\H^1(\PP^{14},\ki_{\gr}^k\cdot \ki_{\PP^8}(m) )=\H_a^{1}(I^k\cap J)_m$ can be computed in terms of local cohomology, we are left to prove that the latter is zero.
	By \cite[Theorem A]{perlman} the Castelnuovo-Mumford regularity $\reg(I^m)=2m$.
	By \cite[Proposition 1.2]{con_herz} we have that $\reg(I^m/J.I^m)=\reg(I^m)=2m$.
	The short exact sequence
	\[
	0\to J.I^m\to I^m\to I^m/J.I^m\to 0
	\]
	gives us $\reg(J.I^m)\leq 2m+1$.
	By definition of Castelnuovo Mumford regularity we obtain $\H^i_a(J.I^d)_{n-i}=0$ for $n\geq 2m +2$ hence the claim.

	For the second statement we want to prove that 
	$\Sym^k\H^0(\bl_S\PP^8,2H-E)\to \H^0(\bl_S\PP^8, k(2H-E))$ is surjective.
	We observe that by the first part of the proof 
	the left vertical arrow of the following diagram is an isomorphism
	\[
	\xymatrix{
		\Sym^k\H^0(\bl_{\gr}\PP^{14},2H-E)\ar[r]\ar[d]& \H^0(\bl_{\gr}\PP^{14}, k(2H-E))\ar[d]\\
		\Sym^k\H^0(\bl_S\PP^8,2H-E)\ar[r]& \H^0(\bl_S\PP^8, k(2H-E)).
	}
	\]
	Moreover the linear system $|2H-E|$ induces a surjective morphism on $\bl_{\gr}\PP^{14}\to \PP(\H^0(\bl_{\gr}\PP^{14},2H-E))$ this means that the top horizontal map is an isomorphism as well.
	The left vertical map is surjective by the first part of the proof.
	
\end{proof}

\begin{lem}\label{lem_alg_fibre_space}
	Consider an algebraic fibre space $f\colon N\to Y$, where $N$ is a smooth projective variety and $Z\subset N$ a smooth closed subvariety such that the restriction $f_|\colon Z\to f(Z)$ has connected fibres.
	We denote by $\ko(1)$ a very ample line bundle on $Y$, $T:=f^*\ko(1)$.
	If the restriction morphism $\H^0(N,T)\to \H^0(Z,T_|)$ is surjective and the graded algebra $\bigoplus_{k\geq 0}\H^0(Z,T_{|}^k)$ is generated in degree one,  then $f(Z)$ is normal and $f_|\colon Z\to f(Z)$ is an algebraic fibre space with a natural embedding $f(Z)\subseteq \PP(\H^0(Z,T_|))$.
\end{lem}
\begin{proof}
	Since $Z$ is normal, $f_|\colon Z\to f(Z)$ factors trough the normalization of the image $\nu \colon \widetilde{f(Z)}\to f(Z)$. We consider the Stein factorization of the obtained morphism $Z\to \widetilde{f(Z)}$ and we denote it by 
	\[
	\xymatrix{
		&Z\ar[dl]^{\phi}\ar[dd]^{f_|}&\\
		Z'\ar[d]^{\mu}&&\\
		\widetilde{f(Z)}\ar[r]^{\nu}&f(Z)&
	}
	\]
	where $\phi$ is an algebraic fibre space and $\mu$ a finite morphism.
	Since $f_|$ has connected fibres we have that $\mu,\nu$ are bijections,
	this proves that $\mu$ is an isomorphism and we assume $\mu=\id$.
	We observe that $\nu^*\ko(1)$ is ample hence for $k\gg1$ the line bundle $\nu^*\ko(k)$ is very ample and $\phi^*(\nu^*\ko(k))$ induces the algebraic space $\phi$, see \cite[Theorem 2.1.27]{lazarsfeld_positivityI}.
	Moreover since the algebra of global sections of $T_|$ is generated in degree one we have that $\Sym^k \H^0(Z,T_|)\to \H^0(Z,T_|^k)$ is surjective.
	In order to prove that the normalization 
	$\nu\colon \widetilde{f(Z)}\to f(Z)$ is an isomorphism
	we consider the commutative diagram
	\[
	\xymatrix{
		\widetilde{f(Z)}\ar[d]^{\nu}\ar[r]&\PP(\H^0(Z,T^k_|))\ar[rd]&\\
		f(Z)\ar[r]&\PP(\H^0(Z,T_|))\ar[r]^{\mathrm{Ver}}&\PP(\Sym^k\H^0(Z,T_|))
	}
	\]
	where the horizontal and diagonal arrows are closed embeddings.
\end{proof}

\begin{prop}\label{prop_f_is_algfibrespace}
    The morphism 
    \[
    f:X=\bl_S{\PP^8}\to \overline{X}
    \]
    is an algebraic fibre space given by the complete linear system of $\ko_{\bl_S\PP^8}(2H-E)$.
    In particular $\overline{X}$ is normal.
\end{prop}
\begin{proof}
By the description of $q\colon \bl_{\gr}\PP^{14}\to {\PP^{14}}^{\vee}$ in Lemma~\ref{lem_blP14} it is clear that $f$ has connected fibres,
hence it is enough to apply Lemmata~\ref{lem_alg_fibre_space} and~\ref{lem_surj_restrH0}.
\end{proof}

\begin{proof}[Proof of Theorem~\ref{thm_geom_descr_contr}]
    As already mentioned, by \cite[Proposition 7.5]{arav24} we have that 
    \[
    \Sigma_1=\bl_S\PP^8\subset \bl_{\gr(2,V_6)}\PP^{14}=\bl_{\gr(2,V_6^{\vee})}{\PP^{14}}^{\vee},
    \]
    the last identification is explained in Lemma~\ref{lem_blP14}.
    Moreover by Proposition~\ref{prop_c1} the morphism ${c_1}_|\colon\Sigma_1\to \overline{\Sigma}_1$ is an algebraic fibre space induced by $L-2E$ and by Proposition~\ref{prop_f_is_algfibrespace} the morphism $f\colon X\to \overline{X}$ is an algebraic fibre space induced by $H-2E$.
    By definition $L=H$ and $X=\bl_S\PP^8$.
    This proves that ${c_1}_|, f$ agree hence $\overline{X}=\overline{\Sigma}_1$.
    To prove the last claim we show that for two distinct points $e_1\wedge e_2, e_3\wedge e_4 \in S$ the image of the secant line $\lambda e_1\wedge e_2+\mu e_3\wedge e_4$ is a point $\phi(\lambda e_1\wedge e_2+\mu e_3\wedge e_4 )=e_1\wedge e_2\wedge e_3\wedge e_4$ 
    hence the strict transform of the line trough $p_| \colon\bl_S\PP^8\to \PP^8$ gets contracted to a point in $\overline{X}$. For dimension reasons we have realized an open subset of $S^{[2]}\subset \overline{X}$.
\end{proof}

\section{The divisorial contraction}\label{Section_contraction}
In this section we describe the divisorial contraction $g\colon \ml\to \overline{M}$ induced by wall-crossing of Bridgeland moduli spaces.
See~\eqref{eq_wallcros} for the description of $\ml$ as well as 
 \cite[Section 3.4]{antisymplI}.
Using some computations from \cite{arav24} we will describe the antisymplectic involution $\taul$ on $\ml$
and the induced contraction $g_|\colon \Sigmal\to g(\Sigmal)$. 
We will prove that $g_|$ is an algebraic fibre space, i.e., $(g_|)_*\ko_{\Sigmal}=\ko_{g(\Sigmal)}$, see Proposition~\ref{prop_barsigm_normal},
and describe its fibres.

We will use the notation $\Sigma=\Sigmal$, $\overline{\Sigma}=g(\Sigmal)$, $M=\ml$, $\tau=\taul$ and 
we use the abbreviation $\H(-,-)$ for $\Hom(-,-)$.
Recall the following
\begin{lem}[{\cite[Lemma 3.23]{antisymplI}}]
    The divisorial contraction 
$g\colon M\to \overline{M}$ contracts a prime irreducible divisor $\Delta$ and we have a stratification  $\Delta=\Delta(1)\sqcup \Delta(2)$
where $\Delta(1)\subset \Delta$ is an open dense subset and $\Delta(2)=\Delta\setminus\Delta(1)$.
Moreover $\Delta(k)$ is isomorphic to an analytically locally trivial Grassmannian bundle $\gr(k,\mathcal{U}_{2k})$ over the $(16-2k^2)$-dimensional moduli space 
$M_{\sigmab}^{st}(-2k,(1+k)h,1-8(1+\frac{k}{2}))$ and where $\mathcal{U}_{2k}$ is a twisted vector bundle of rank $2k$ induced by the relative Ext-sheaf, $k=1,2$.
\end{lem}
We will use the following notation
\begin{equation*}
\begin{split}
    g_1:\Delta(1)\to M_{\sigmab}^{st}(-2,2h,-11)=:M_1\\
    g_2:\Delta(2)\to M_{\sigmab}^{st}(-4,3h,-15)=:M_2.
\end{split}
\end{equation*}
The involution $\tau$ on $\ml$ is induced by the 
anti-autoequivalence 
$\Psi= \clrhom(-,\ko_S(-H)[1])$ and 
the involution $\bar{\tau}$ on $\overline{M}$ restricted to $M_k,k=1,2$ is induced by
$\Phi=\ST_{A}\circ \Psi$, see \cite[Lemma 3.24]{antisymplI}.

The main result of this section can be summarized as follows.

\begin{thm}\label{thm_contr_g}
    The restriction $g_|\colon \Sigma\to  \overline{\Sigma}$ is a divisorial contraction on the normal variety $\overline{\Sigma}$.
    It contracts a $\lgr(4)$-bundle on a smooth cubic 4fold $Y\subset\overline{\Sigma}$ and the singularities of $\overline{\Sigma}$ are locally analytically isomorphic to $\AAA^4\times \mathrm{C}(\lgr(4))$ where the second factor is the affine cone over the Lagrangian Grassmannian $\lgr(4)$.
    Moreover $g_|\colon \Sigma\to \Sigmab$ is the blow-up of $g(\Delta(2)\cap \Sigma)=Y\subset \Sigmab$. 
\end{thm}

\subsection{Local structure at  $\overline{M}\cap g(\Delta(k))$}\label{section_localstr}
In this section we make explicit the result of \cite{arbarellosaccaunupdate} which describes analytic neighborhoods of the singular points of $\overline{M}$, see Proposition~\ref{prop_Mbar_local}.
The study of the functors $\Psi, \Phi$ inducing the antisymplectic involution gives us a description of the component $\overline{\Sigma}\subset\overline{M}$ of the fixed locus of $\taub$, see Lemmata~\ref{lem_inv_R},~\ref{lem_inv_tau1lin}.
In this section we study the general case in the sense that $k$ can be any integer. 
In the following two section we specialize to the cases $k=1$ and $k=2$ respectively.
As the reader may have noticed, in this paper we are studying only K3 surfaces of genus $8$, but it is worth recalling that in \cite{antisymplI,antisymplII} the authors study all even genus K3s and the computations of this section are valid for all of them as well.

Recall that the Lazarsfeld-Mukai bundle $A$ on the K3 surface $S$ is defined, up to isomorphism, by the following short exact sequence
\[
0\to A\to \ko_S^{\oplus 2}\to i_{C,*}\xi\to 0
\]
where $C\in |H|$ is a smooth curve and $\xi\in \Pic^{5}(C)$ with $\h^0(C,\xi)=2$.
Recall also that the contraction $g$ restricted to $\Delta(k)\to M_{k}$ is the $(k,2k)$-Grassmannian bundle of the twisted vector bundle
$\mathcal{U}\to M_k$ with fibre $\Hom(B,A[1])$ over the point $B\in M_k$.
Fix a point $x\in \overline{\Sigma}\cap M_2$.
By \cite[Lemma 3.23]{antisymplI} the point $x$ has as representative of the S-equivalence class  an object of the form $B\oplus A^{\oplus k}$.
Let us fix an isomorphism of functors 
    \[
    \zeta\colon  {\rm id}\to \Phi\circ \Phi
    \]
as in \cite[Equation (53)]{arav24}
and isomorphisms 
$\mu: \Phi(A)\to A, \lambda:\Phi(B)\to B$.
We use the following notation $V:=\Hom(B,A[1])$  and $R:=\Hom(B,B[1])$, they are of dimension  $2k$ and  $2(8-k^2)$ respectively.
We consider the following linear map
\begin{equation*}
    \begin{split}
        \hat{\Phi}:\Hom(B,A[1])&\to \Hom(A,B[1])\\
        \alpha&\mapsto (\lambda\circ\Phi(\alpha)\circ\mu^{-1}[-1])[1]
    \end{split}
\end{equation*}
and the induced
bilinear form
\begin{equation*}
    \begin{split}
        h:\H(B,A[1])\times\H(B,A[1])\overset{\hat{\Phi}\times\id}{\longrightarrow}\H(A,B[1])\times \H(B,A[1])\overset{-\circ( -[-1])}{\longrightarrow}&\H(A[-1],A[1])
    \end{split}
\end{equation*}
explicitly $h(a,b)=b\circ\lambda\circ \Phi(a)\circ \mu^{-1}[1]$.
By \cite[Propositioin 11.3]{arav24} $h$ is symplectic, and by Serre duality
the paring
\[
\begin{split}
    \Hom(B,A[1])\times\Hom(A,B[1])&\overset{-\circ( -[-1])}{\longrightarrow} \Hom(A[-1],A[1])\\
    (a,b)&\mapsto  b\circ (a[-1])
\end{split}
\]
identifies $\Hom(A,B[1])\cong V^{\vee}$ hence we get 
a linear map
\begin{equation*}
\begin{split}
    \tau:V&\to V^{\vee}\\
\end{split}
\end{equation*}
such that
\begin{equation}
    h=\langle\tau(-),-\rangle: V\times V\overset{\tau\times{\rm id} }{\longrightarrow}V^{\vee}\times V \overset{\langle-,-\rangle}{\longrightarrow} \CC
\end{equation}
where
$\langle-,-\rangle$ is the dual pairing.
In a symplectic basis $e_1,\dots, e_k,f_1,\dots,f_k$ of $V$ we have that
$h$ is given by the matrix
\[
\Omega=
\begin{pmatrix}
0&\id_k\\
-\id_k&0
\end{pmatrix}.
\]
We get that the map
$\tau:V\to V^{\vee}$ is given by $-\Omega$,
where on $V^{\vee}$ we used the dual basis, i.e. $\tau(e_i)=f_i^{\vee}, \tau(f_i)=-e_i^{\vee}$.

We will use the notation $A^{\oplus k}=A\otimes W$ where $W$ is a $k$-dimensional vector space.
Let us fix a basis $W=\langle w_1,\dots, w_k\rangle$.
When we consider 
$\Ext^1(x,x)=(V\otimes W)\oplus (V^{\vee}\otimes W^{\vee})\oplus R$
and write an element of 
$u\in V\otimes W\oplus V^{\vee}\otimes W^{\vee}$ as a $2k\times 2k$ matrix of the form
\[
u=\begin{pmatrix}
    X&Y\\
    S&T
\end{pmatrix}
\]
with $X,Y,S,T\in \Mat(k\times k,\CC)$, more precisely 
\begin{equation}\label{eq_def_coordinates}
u=\left(\sum x_{ij}w_i\otimes e_j\right)\oplus\left(\sum y_{ij}w_i\otimes f_j\right)\oplus\left(\sum s_{ij}w_i^{\vee}\otimes {e_j^{\vee}}\right)\oplus\left(\sum t_{ij}w_i^{\vee}\otimes {f_j^{\vee}}\right).
\end{equation}
The action of an element $N\in \GL(W)=\Aut(A\otimes W)$ on $\Ext^1(x,x)$ is given by 
\[
Nu=
\begin{pmatrix}
    NX&NY\\
    \tensor[^t]{N}{^{-1}}S&\tensor[^t]{N}{^{-1}}T
\end{pmatrix}.
\]
The singularity of $\overline{M}$ is described in terms of the quiver and the Yoneda pairing is given by the following momentum map
\[
\begin{split}
    \mu:\Ext^1(x,x)=V\otimes W\oplus V^{\vee}\otimes W^{\vee}\oplus R&\to \Ext^2(B,B)\oplus W\otimes W^{\vee}\\
    \left(\sum v_i\otimes w_i\right)\oplus\left(\sum v_i^{\vee}\otimes w'_i\right)\oplus r&\mapsto \left(\sum w'_i(w_i)\right)\oplus\left(\sum w_i\otimes w_i'\right)
\end{split}
\]
where $v_1,\dots, v_{2k}$ is a basis of $V$, see \cite[Corollary 4.1]{arbarellosaccaunupdate}.
In coordinates
\begin{equation}\label{eq_def_maldetesta}
\begin{split}
    \mu:\Ext^1(x,x)&\to \End(\CC)\oplus\End(W)\\
    u\oplus r&\mapsto \left(\sum x_{ij}s_{ij}+ y_{ij}t_{ij}\right)\oplus \left(X\cdot\tensor[^t]{S}{}+Y\cdot\tensor[^t]{T}{}\right).
\end{split}
\end{equation}
We observe that $\mu(u,r)=\mu_1(u)$ where 
\[
\mu_1: V\otimes W\oplus V^{\vee}\otimes W^{\vee}\to \End(\CC)\oplus\End(W)
\]
is the restriction of $\mu$.

Let us denote by $G:=\CC^*\times\GL(W)$ and let us consider the natural action on $\Ext^1(x,x)$, it is equivariant with respect to $\mu$.
We define 
\[
\fM:=\fM_1\times R
\]
where $\fM_1:=\mu_1^{-1}(0)//\GL(W).$
We have the following
\begin{prop}[{\cite[Proposition 3.1, Corollary 4.1]{arbarellosaccaunupdate}}]
    The moduli space $\overline{M}$
    is locally analytically at $x$ isomorphic to a neighborhood of 
    $0\in\fM$.
\end{prop}

We will now study how the involution $\taub$ in a neighborhood of a singular point of $\overline{M}$. Recall also that by \cite[Lemma 3.24]{antisymplII} $\bar{\tau}$ on $\overline{M}$ restricts to $M_k\subset \overline{M}$, where it is given by the functor $\Phi$. 
Hence the linear action on $\Ext^1(x,x)$ is induced by the linear map
\begin{equation}\label{eq_def_taulin}
    \tau^l:\Hom(x,x[1])\overset{\Phi}{\longrightarrow}\Hom(\Phi(x[1]),\Phi(x))\cong \Hom(x,x[1]).
\end{equation}
Let us define the last isomorphism in~\eqref{eq_def_taulin}: we consider the direct sums
\[
\begin{split}
\Hom(\Phi(x[1]),\Phi(x))&=\Hom(\Phi(B[1]),\Phi(B))\oplus\Hom(\Phi(B[1]),\Phi(A))^{\oplus 2}\oplus \Hom(\Phi(A[1]),\Phi(B))^{\oplus 2}\\
\Hom(x,x[1])&=\Hom(B,B[1])\oplus \Hom(B,A[1])^{\oplus 2}\oplus\Hom(A,B[1])^{\oplus 2}
\end{split}
\]
we identify the first summands by 
\[
\begin{split}
    \Hom(\Phi(B[1]),\Phi(B))&\to \Hom(B,B[1])\\
    \alpha&\mapsto \lambda[1]\circ(\alpha[1])\circ\lambda^{-1}
\end{split}
\]
the other identifications are analogous.
In particular $\Phi$ induces the following linear maps
\begin{equation*}
    \begin{split}
        \hat{\Phi}:\Hom(B,A[1])&\to \Hom(A,B[1])\\
        \alpha&\mapsto (\lambda\circ\Phi(\alpha)\circ\mu^{-1}[-1])[1]
    \end{split}
\end{equation*}
and 
\begin{equation*}
    \begin{split}
        \tilde{\Phi}:\Hom(A,B[1])&\to \Hom(B,A[1])\\
        \beta&\mapsto \mu[1]\circ\Phi(\beta[-1])\circ\lambda^{-1}.
    \end{split}
\end{equation*}
Using the decomposition $\Hom(x,x[1])=W\otimes V\oplus V^{\vee}\otimes W^{\vee}$ the linear map $\tau^l$ can be written as a sum $\tau^l:=\tau^l_1+\tau_2^l$ as follows 
\begin{equation*}
    \begin{split}
        \tau^l_1\colon W\otimes V\oplus V^{\vee}\otimes W^{\vee} &\to W\otimes V\oplus V^{\vee}\otimes W^{\vee}\\
        \tau^l_2\colon R&\to R.
    \end{split}
\end{equation*}
We also observe that
\begin{equation*}
    \tau_1^l=\hat{\Phi}^{\oplus k}\oplus\tilde{\Phi}^{\oplus k}.
\end{equation*}

\begin{lem}[{\cite[Lemmata 4.1 and 11.5]{arav24}}]\label{lem_annib4.1_11.5}
    The functor $\Phi$ acts on $\Hom(\Phi(A),A)$ as $-1$.
\end{lem}
\begin{proof}
    Consider \cite[Equation (57)]{arav24} with $F=A$ and $G=\ST'_{\Psi(A)}(A)$ and \cite[Equation (55)]{arav24} with $F=A$ and compose them to obtain 
    \begin{equation}
        \xymatrix{
        \H(\Phi(A),A)\ar[r]^{\zeta_A^{-1}\circ\Phi(-)}\ar[d]&\H(\Phi(A),A)\ar[d]\\
        \H(\ko(-H)[1],A\otimes \ST'_{\Psi(A)}(A))\ar[r]^{\nu\circ-}&\H(\ko(-H)[1],ST'_{\Psi(A)}\otimes A)
        }
    \end{equation}
where $\nu: A\otimes B\to B\otimes A$ acts as $a\otimes b\mapsto (-1)^{pq}b\otimes a$ for $a\in A^{p}, b\in B^{q}$.
Using the identification $\ST'_{\Psi(A)}(A)\cong A[1]$, we get that $\nu(a\otimes b)=b\otimes a$ for $a\in A, b\in (A[1])_{-1}$.
We now consider the following inclusion
\begin{equation}
    \H(\ko(-H)[1], (\wedge^2 A)[1])\to \H(\ko(-H)[1],A\otimes (A[1]))
\end{equation}
we have that 
\[
\dim (\H(\ko(-H)[1], (\wedge^2 A)[1]))=\H^0(S,\ko_S)=1=\dim(\H(\ko(-H)[1],A\otimes (A[1]))),
\]
hence it is an isomorphism.
This shows that for any $\phi\in \H(\ko(-H)[1],A\otimes (A[1]))$ and $s$ a local section of $\ko(-H)[1]$ we have 
\[
\nu\circ \phi(s)=\nu\left(\sum_i \phi_{1,i}(s)\otimes\phi_{2,i}(s)\right)=\sum_i \phi_{2,i}(s)\otimes \phi_{1,i}(s)=- \phi(s),
\]
hence the claim follows.
\end{proof}

We will now give the explicit expression for $\tau_1^l$.
\begin{prop}\label{prop_Mbar_local}
    With respect to the basis defined in~\eqref{eq_def_coordinates} we have that 
    \[
    \tau^l_1\left(
    \begin{pmatrix}
        X&Y\\
        S&T
    \end{pmatrix}\right)=
    \begin{pmatrix}
        -T&S\\
        -Y&X
    \end{pmatrix}
    \]
\end{prop}
\begin{proof}
    As noted above it is enough to study $\hat{\Phi}$ and $\tilde{\Phi}$.
    Consider the following diagram
    \begin{equation}\label{eq_diag_hlapecjernej}
        \xymatrix{
        \H(B,A[1])\times\H(B,A[1])\ar[d]^{\hat{\Phi}\times\hat{\Phi}}\ar[r]^{\hat{\Phi}\times\id}&\H(A,B[1])\times \H(B,A[1])\ar[d]^{\tilde{\Phi}\times\hat{\Phi}}\ar[r]^{P_1}&\H(A[-1],A[1])\ar[d]^{L}\\
        \H(A,B[1])\times\H(A,B[1])\ar[r]^{\tilde{\Phi}\times\id}&\H(B,A[1])\times\H(A,B[1])\ar[r]^{P_2}&\H(A[-1],A[1])
        }
    \end{equation}
where $P_1,P_2$ are defined as follows
\begin{equation*}
    \begin{split}
        P_1: \H(A,B[1])\times \H(B,A[1])&\to\H(A[-1],A[1])\\
        (a,b)&\mapsto b\circ a[-1]
    \end{split}
\end{equation*}
\begin{equation*}
    \begin{split}
        P_2:  \H(B,A[1])\times\H(A,B[1])&\to\H(A[-1],A[1])\\
        (a,b)&\mapsto a\circ b[-1]
    \end{split}
\end{equation*}
and 
\begin{equation*}
    L(\alpha):=\mu[1]\circ\Phi(\alpha)\circ(\mu^{-1}[-1]).
\end{equation*}
By direct computations we get that \eqref{eq_diag_hlapecjernej} is commutative.
We now prove that $L$ is the identity
\begin{equation}
    \begin{split}
        L(\alpha)&=\mu[1]\circ\Phi(\alpha)\circ(\mu^{-1}[-1])=\\
        &=\mu[1]\circ\Phi(\alpha\circ\mu[-1]\circ\mu^{-1}[-1])\circ(\mu^{-1}[-1])\\
        &=\mu[1]\circ\Phi(\mu^{-1}[-1])\circ\Phi(\alpha\circ\mu[-1])\circ(\mu^{-1}[-1])\\
        &=\mu[1]\circ\Phi(\mu^{-1}[-1])\circ\zeta_{A[1]}\circ\zeta^{-1}_{A[1]}\circ\Phi(\alpha\circ\mu[-1])\circ(\mu^{-1}[-1])\\
        &\overset{\bullet}{=}-\mu[1]\circ\Phi(\mu^{-1}[-1])\circ\zeta_{A[1]}\circ\alpha\circ\mu[-1]\circ(\mu^{-1}[-1])\\
        &=-\mu[1]\circ\Phi(\mu^{-1}[-1])\circ\zeta_{A[1]}\circ\alpha
    \end{split}
\end{equation}
where the equality $\overset{\bullet}{=}$ follows from \cite[Lemma 11.5]{arav24}.
By Lemma~\ref{lem_annib4.1_11.5} we get that 
\[
\zeta^{-1}_A\circ \Phi(\mu)=-\mu
\]
hence
\[
\mu[1]\circ\Phi(\mu^{-1}[-1])\circ\zeta_{A[1]}=\mu[1]\circ\Phi(\mu)^{-1}[1]\circ\zeta_{A[1]}=(\mu\circ\Phi(\mu)^{-1}\circ\zeta_{A})[1]=-\id.
\]
We deduce that $L=\id$.
Using the identification $V^{\vee}\cong \H(A,B[1])$ given by $P_1$ and denoting by $\tilde{\tau}:V^{\vee}\to V$ the morphism induced by $\tilde{\Phi}$, the diagram \eqref{eq_diag_hlapecjernej} becomes

\begin{equation}\label{eq_diag_hlapecjernej2}
        \xymatrix{
        V\times V\ar[d]^{\tau\times \tau}\ar[r]^{\tau \times\id}&V^{\vee}\times V\ar[d]^{\tilde{\tau}\times\tau}\ar[r]^{\langle-,-\rangle_{V^{\vee},V}}&\H(A[-1],A[1])\ar[d]^{L}\\
        V^{\vee}\times V^{\vee}\ar[r]^{\tilde{\tau}\times\id}&V\times V^{\vee}\ar[r]^{\langle-,-\rangle_{V,V^{\vee}}}&\H(A[-1],A[1])
        }
    \end{equation}
and we get
\[
\delta_{ij}=h(e_i,f_j)=\langle\tau(e_i),f_j\rangle_{V^{\vee},V}=\langle f_j,\tau(e_i)\rangle_{V,V^{\vee}}=\langle \tilde{\tau}(\tau(e_i)),\tau(f_j)\rangle_{V,V^{\vee}}=\langle \tilde{\tau}(f_i^{\vee}),-e_j^{\vee}\rangle_{V,V^{\vee}}
\]
hence
\[
\tilde{\tau}(f_i^{\vee})=-e_i
\]
similarly $\tilde{\tau}(e_i^{\vee})=f_i$.
\end{proof}

\begin{lem}\label{lem_inv_R}
    The linear map $\tau^2_l:R\to R$ is an involution with smooth fixed locus.
\end{lem}
\begin{proof}
    We know by \cite[Lemma 3.24]{antisymplI} that $\Phi$ induces on the smooth variety $M_k$ an involution hence the induced map $\tau^2_l$ on the tangent space $R$ is an involution with smooth fixed locus.
\end{proof}

\begin{lem}\label{lem_inv_tau1lin}
    The order of $\tau^1_l\colon W\otimes V\oplus V^{\vee}\otimes W^{\vee} \to W\otimes V\oplus V^{\vee}\otimes W^{\vee} $ is four and it induces an involution on $\fM_1$.
\end{lem}
\begin{proof}
Direct computations gives
\[
\tau_l^1\circ \tau_l^1=-\id.
\]
We only need to show that $\tau^l_1$ is equivariant with respect to the action of $\GL(k)$.
    \[
    \begin{split}
        \tau^l_1\left(M\begin{pmatrix}
            X&Y\\
            S&T
        \end{pmatrix}\right)
        =
        \tau\left( 
        \begin{pmatrix}
            MX&MY\\
            \tensor[^t]{M}{^{-1}}S&\tensor[^t]{M}{^{-1}}T
        \end{pmatrix}
        \right)
        =
        \begin{pmatrix}
            -\tensor[^t]{M}{^{-1}}T&\tensor[^t]{M}{^{-1}}S\\
            -MY&MX
        \end{pmatrix}\\
        =
        \tensor[^t]{M}{^{-1}}
        \left(\tau^l_1\left(
        \begin{pmatrix}
            X&Y\\
            S&T
        \end{pmatrix}
        \right)\right).
    \end{split}
    \]
\end{proof}

\subsection{Local structure at  $\overline{M}\cap c(\Delta(1))$}

In this section we study the singularity of $\overline{M}$ at a point $x\in g(\Delta(1))$.  
The map $\mu$ defined at~\eqref{eq_def_maldetesta} is
\[
\begin{split}
    \mu: V\oplus V^{\vee}\oplus R&\to \CC\oplus \CC\\
    u\oplus r&\to (xs+yt,xs+yt).
\end{split}
\]
Hence the analytic neighborhood of $x\in c(\Delta(1))$ is $\fM_1\times R$
where $R=\AAA^{2(8-k^2)}$ and 
\[
\fM_1=\{xs+yt=0\}//\CC^*
\]
where $\CC^*\ni \lambda$ acts as $\lambda(x,y,s,t)=(\lambda x,\lambda y,\lambda^{-1}s,\lambda^{-1}t)$.
The local equations are given as follows.
Recall that $\AAA^4//\CC^*$ is
\[
\CC[a,b,c,d]/(ad-bc)
\]
where the quotient morphism is given by $a\mapsto xs, b\mapsto xt,c\mapsto ys, d\mapsto yt$
and $xs+ty=0$ becomes $a+d=0$.
Hence
\[
\fM_1=\CC[a,b,c]/(a^2+bc)
\]
that is an $A_1$ surface singularity.
We deduce  the following
\begin{lem}\label{lem_bSigma_Delta1_empt}
    The fixed locus $\Sigma$ does not intersect $\Delta(1)$.
\end{lem}
\begin{proof}
    Suppose by contradiction it does and consider a point $x\in \Sigmab\cap c(\Delta(1))$.
    By the description above we have an analytic neighborhood of $x$ of the form
    \[
\CC[a,b,c]/(a^2+bc)\times
\AAA^{14}.
    \]
    As already recalled in the previous section $\bar{\tau}$ is the product of two involutions on each of the two factors above. 
    Moreover on the factor 
    $\CC[a,b,c]/(a^2+bc)$
    the action has as fixed locus the image of 
    \[
    \left\{u=
    \begin{pmatrix}
        x&y\\
        s&t
    \end{pmatrix}
    \text{ such that } 
    \begin{pmatrix}
        -t&s\\
        -y&x
    \end{pmatrix}
    =
    \begin{pmatrix}
        \lambda x&\lambda y\\
        \lambda^{-1}s&\lambda^{-1} t
    \end{pmatrix}
    \text{ for some }\lambda\in \CC^*
    \right\}=\{0\}
    \]
    trough the quotient map.
    This proves that an analytic local neighborhood of $\Sigmab$ is $0\times \AAA^{7}$, but $\Sigmab$ has dimension $8$ hence a contradiction.
\end{proof}

\subsection{Local structure at  $\overline{M}\cap c(\Delta(2))$}
In this section we study the singularity of $\overline{M}$ at a point $x\in g(\Delta(2))$.
Recall that the contraction $g\colon M\to \overline{M}$ restricted to $\Delta(2)\to M_{2}:=M_{\sigmab}(-4,-3h,-15)$ is the $(2,4)$-Grassmannian bundle of the twisted vector bundle
$\mathcal{U}\to M_2$ with fibre $\Hom(B,A[1])$ over the point $B\in M_2$.

Fix a point $x=[B\oplus A^{\oplus2}]\in \Sigmab\cap M_2$, by Proposition~\ref{prop_Mbar_local} an analytic neighborhood of $x$
is described as follows
\begin{equation}\label{eq_1}
    \fM:=\fM_1\times R
\end{equation}
where $\fM_1:=\mu_1^{-1}(0)//\GL(W)$ and $\mu_1$ is defined at~\eqref{eq_def_maldetesta}.
Consider the involution $\bar{\tau}$ in an anlytic neighborhood $x\in \fM$. We observe that $\Sigmab$ is a connected component of the fixed locus with the reduced scheme structure, since the involution splits on the product~\eqref{eq_1}, we have $\Sigmab=\Sigmab_1\times \Sigmab_R$.
Let us consider the preimage $\eta:=\pi^{-1}(\Sigmab_1)$ of $\Sigmab_1$  via the quotient morphism 
\[
\pi: \{\mu_1=0\}\to \fM_1.
\]
It is clear that 
\begin{equation}
    \eta=
    \left\{
    u=\begin{pmatrix}
        X&Y\\
        NY&-NX
    \end{pmatrix}
    :\; X,Y\in \Mat_{2\times 2}(\CC)\textit{ and } N\in \GL(2)
    \right\}.
\end{equation}
Recalling the description of the involution $\taub$ given in Lemma~\ref{lem_inv_tau1lin} we see that in order to study the quotient $\overline{\Sigma}$ we need first
solve the following equation
\[
\begin{pmatrix}
    -T&S\\
    -Y&X
\end{pmatrix}
=
\begin{pmatrix}
    NX&NY\\
    \tensor[^t]{N}{^{-1}}S&\tensor[^t]{N}{^{-1}}T
\end{pmatrix}
\]
that give us
\[
N=-^{t}{N}{}, \; \; S=NY, \; -T=NX.
\]
This motivates the definition of
 the following two subspaces
\begin{equation}\label{eq_def_Q}
    Q:=\left\{
    \begin{pmatrix}
        X&Y\\
        NY&-NX
    \end{pmatrix}
    :
    N\in \GL(2) \;\text{ with } -N=\tensor[^t]{N}{}, \;\;X,Y\in \Mat_{2\times 2}(\CC)
    \right\}
\end{equation}
\begin{equation}
    F:=\left\{
    \begin{pmatrix}
        X&Y\\
        \Pi Y&-\Pi X
    \end{pmatrix}
    :
    X,Y\in \Mat_{2\times2}(\CC)
    \right\}
\end{equation}
where 
\[
\Pi=\begin{pmatrix}
    0&1\\
    -1&0
\end{pmatrix}.
\]
\begin{lem}\label{lem_simplecomp}
    The subset $F\subset V\otimes W\oplus V^{\vee}\otimes W^{\vee}$ is a linear (smooth) closed subset of dimension 8, 
    \begin{equation}
        \GL(W)F=Q=\left\{
    \begin{pmatrix}
        0&\lambda\\
        -\lambda&0
    \end{pmatrix}
    u| u\in F, \lambda\in \CC^*
    \right\}.
    \end{equation}
    Moreover the matrices that fixes $F$ are
    \[
    \Stab_{\GL(W)}(F)=
    \left\{
    M\in \GL(W)| \prescript{t}{}{M}^{-1}\Pi M^{-1}=\Pi
    \right\}=\Sp(1)=\SL(2).
    \]
\end{lem}
\begin{proof}
    Let us show the first equality: it is enough to note that for an invertible matrix $N\in \GL(W)$ we have 
    \[
    N\begin{pmatrix}
        X&Y\\
        \Pi Y&-\Pi X
    \end{pmatrix}=
    \begin{pmatrix}
        N X&N Y\\
        \prescript{t}{}{N}^{-1}\Pi N^{-1}N Y&-\prescript{t}{}{N}^{-1}\Pi N^{-1}N X
    \end{pmatrix}
    \]
    where $^{t}\left(\prescript{t}{}{N}^{-1}\Pi N^{-1}\right)=-\prescript{t}{}{N}^{-1}\Pi N^{-1}$.
    Conversely we have that for 
    \[
    u=\begin{pmatrix}
        X&Y\\
        NY&-NX
    \end{pmatrix}\in Q
    \]
we have $N=\lambda\Pi,\lambda\in \CC^*$
hence
\[
 u=\begin{pmatrix}
     \lambda^{-\frac{1}{2}}\lambda^{\frac{1}{2}} X&\lambda^{-\frac{1}{2}}\lambda^{\frac{1}{2}}Y\\
     \lambda^{\frac{1}{2}}\Pi\lambda^{\frac{1}{2}}Y&-\lambda^{\frac{1}{2}}\Pi\lambda^{\frac{1}{2}}X
 \end{pmatrix}
 =\lambda^{-\frac{1}{2}}
 \begin{pmatrix}
     \lambda^{\frac{1}{2}} X&\lambda^{\frac{1}{2}}Y\\
     \Pi\lambda^{\frac{1}{2}}Y&-\Pi\lambda^{\frac{1}{2}}X
 \end{pmatrix}.
 \]
The other follows from similar computations.
\end{proof}

\begin{lem}\label{lem_F}
    The intersection
    \[
    F\cap \mu_1^{-1}(0)
    \]
    is a quadric hypersurface in $\Mat_{2\times 4}(\CC)$, with the origin as the only singular point, in particular it is Cohen-Macaulay and normal of dimension $7$.
    It is defined by the equation
    \begin{equation}
        x_{11}y_{21}-x_{21}y_{11}+x_{12}y_{22}-x_{22}y_{12}=0.
    \end{equation}
\end{lem}
\begin{proof}
    Consider 
    $u=\begin{pmatrix}
        X&Y\\
        \Pi Y&-\Pi X
    \end{pmatrix}\in F$ ,
    we will write $x=(x_{11},x_{12},x_{21},x_{22})$ and similarly we will consider the vector $y$. 
    Let us denote by $q$ the bilinear form 
    \[q(X,Y):=\prescript{t}{}{x}\cdot\Omega \cdot y=x_{11}y_{21}-x_{21}y_{11}+x_{12}y_{22}-x_{22}y_{12}\]
    By direct computations we get
    \[
    \begin{split}
        \mu_1(u)&=(2q(X,Y), (X\cdot\prescript{t}{}{Y}-Y\cdot\prescript{t}{}{X})\prescript{t}{}{\Pi})\\
        &=\left(2q(X,Y), \; \begin{pmatrix}
            0&q(X,Y)\\
            -q(X,Y)&0
        \end{pmatrix}
        \prescript{t}{}{\Pi}
        \right).
    \end{split}
    \]
Hence $\mu_1^{-1}(0)\cap F=F\cap\{q=0\}$, hence we study the map quadratic form 
\[
q: F\to \CC
\]
the differential is clearly
\[
dq_{(X_0,Y_0)}(X,Y)=(\prescript{t}{}{x}\Omega y_0+\prescript{t}{}{x_0}\Omega y)
\]
and it is surjective for $(X_0,Y_0)\neq(0,0)$.
Hence $F\cap\mu_1^{-1}(0)\setminus\{0\}$ is smooth of dimension $7$, defined by the equation
\[
x_{11}y_{21}-x_{21}y_{11}+x_{12}y_{22}-x_{22}y_{12}=0.
\]
We conclude that $F\cap\mu_1^{-1}(0)$ is complete intersection, hence Cohen-Macaulay and normal.

\end{proof}

\begin{cor}\label{cor_fXnormal}
    The variety $\fX:=(F\cap\mu_1^{-1}(0))//\Sp(1)$ is normal.
\end{cor}

We will now study the subset $Q\subset (V\otimes W)\oplus(W\otimes V)^{\vee}$, defined at~\eqref{eq_def_Q}, and its closure.
\begin{prop}\label{prop_eq_ofQbar}
    Let us define the following matrix
\begin{equation}
N:=
\begin{pmatrix}
x_{11}&x_{12}&x_{21}&x_{22}&y_{11}&y_{12}&y_{21}&y_{22}\\
t_{21}&t_{22}&-t_{11}&-t_{12}&-s_{21}&-s_{22}&s_{11}&s_{12}
\end{pmatrix}.
\end{equation}
Then closure of $Q\subset (V\otimes W)\oplus(W\otimes V)^{\vee}$, defined in~\eqref{eq_def_Q}, is a determinantal variety defined by $\bar{Q}=\{\rk(N)\leq 1\}$.
\end{prop}
\begin{proof}
    As already noted in Lemma~\ref{lem_simplecomp} we have that any $u\in Q$ is of the form 
    \[
    u=
    \begin{pmatrix}
        X&Y\\
        \lambda\Pi Y&-\lambda\Pi X
    \end{pmatrix}
    \]
    with $\lambda\in \CC^*$
    hence $u\in \{\rk(N)\leq 1\}$.
    The latter is irreducible of dimension 9, see e.g. \cite{arbarello_curvesI}[Chapter II].
\end{proof}
\begin{cor}
    The variety $\bar{Q}$ is Cohen-Macaulay of dimension $9$ and the singular locus is the isolated point $0\in \bar{Q}$, 
    in particular it is normal.
\end{cor}
\begin{proof}
By Proposition~\ref{prop_eq_ofQbar} $\bar{Q}$ is a generic determinantal variety.
By \cite{arbarello_curvesI}[Chapter II] 
$\bar{Q}$ is Cohen-Macaualy and 
$\sing(\bar{Q})=\{\rk(N)=0\}$.
\end{proof}
Let us define
\begin{equation}
    Q_1:=\left\{
    \begin{pmatrix}
        0&0\\
        S&T
    \end{pmatrix}
    |
    \;S,T\in \Mat_{2\times 2}(\CC)
    \right\}
\end{equation}
and similarly
\begin{equation}
    Q_2:=\left\{
    \begin{pmatrix}
        X&Y\\
        0&0
    \end{pmatrix}
    |
    \;X,Y\in \Mat_{2\times 2}(\CC)
    \right\}
\end{equation}
\begin{lem}
    We have that
    \[
    \bar{Q}=Q\cup Q_1\cup Q_2
    \]
    and $Q\cap Q_1=Q\cap Q_2=\{0\}$.
    Hence 
    \[\bar{Q}\cap\mu^{-1}_1(0)=Q_1\cup Q_2 \cup (Q\cap \mu_1^{-1}(0))\]
    and 
    $Q\cap \mu_1^{-1}(0)\subset (V\otimes W)\oplus(V\otimes W)^{\vee}$ is a closed subset.
\end{lem}

\begin{lem}\label{lem_Q3isS3}
    The closed irreducible subvariety $Q_3:=Q\cap\mu_1^{-1}(0)\subset (V\otimes W)\oplus(W\otimes V)^{\vee}$ satisfies Serre's condition $S_3$ and has $0$ as its only singular point.
\end{lem}
\begin{proof}
We note that $Q_3\setminus\{0\}\iso ((F\setminus\{0\})\cap \mu_1^{-1}(0))\times\CC^*$ hence it is smooth by Lemma~\ref{lem_F}.
We prove that the local ring at $0$ satisfies Serre's condition $S_3$. 
The equations 
    $\mu_1=0$ are
    \begin{equation}
        \begin{split}
            &x_{11}  s_{11} + x_{12}  s_{12} + y_{11}  t_{11} + y_{12}  t_{12}, \\
&x_{21}  s_{21} + x_{22}  s_{22} + y_{21}  t_{21} + y_{22}  t_{22}, \\
&x_{11}  s_{21} + x_{12}  s_{22} + y_{11}  t_{21} + y_{12}  t_{22}, \\
&x_{21}  s_{11} + x_{22}  s_{12} + y_{21}  t_{11} + y_{22}  t_{12}.
        \end{split}
    \end{equation}
Macaulay2 computes the ideal of  $\bar{Q}\cap \{\mu_1=0\}$
\begin{equation}
\begin{split}
&(t_{22}, t_{21}, t_{12}, t_{11}, s_{22}, s_{21}, s_{12}, s_{11}) \\
&\cap(y_{22}, y_{21}, y_{12}, y_{11}, x_{22}, x_{21}, x_{12}, x_{11}) \\
&\cap(
x_{12}t_{21} - x_{11}t_{22},\ 
x_{12}t_{12} + x_{22}t_{22},\ 
x_{11}t_{12} + x_{22}t_{21},\\
&\quad s_{21}t_{11} + s_{22}t_{12} - s_{11}t_{21} - s_{12}t_{22},\ 
y_{11}t_{11} + y_{12}t_{12} + y_{21}t_{21} + y_{22}t_{22},\\
&\quad x_{22}t_{11} - x_{21}t_{12},\ 
x_{12}t_{11} + x_{21}t_{22},\ 
x_{11}t_{11} + x_{21}t_{21},\\
&\quad x_{22}s_{22} - y_{12}t_{12},\ 
x_{21}s_{22} - y_{12}t_{11},\ 
x_{12}s_{22} + y_{12}t_{22},\ 
x_{11}s_{22} + y_{12}t_{21},\\
&\quad y_{12}s_{21} - y_{11}s_{22},\ 
x_{22}s_{21} - y_{11}t_{12},\ 
x_{21}s_{21} + y_{12}t_{12} + y_{21}t_{21} + y_{22}t_{22},\\
&\quad x_{12}s_{21} + y_{11}t_{22},\ 
x_{11}s_{21} + y_{11}t_{21},\ 
y_{12}s_{12} + y_{22}s_{22},\ 
y_{11}s_{12} + y_{22}s_{21},\\
&\quad x_{22}s_{12} + y_{22}t_{12},\ 
x_{21}s_{12} + y_{22}t_{11},\ 
x_{12}s_{12} - y_{22}t_{22},\ 
x_{11}s_{12} - y_{22}t_{21},\\
&\quad y_{22}s_{11} - y_{21}s_{12},\ 
y_{12}s_{11} + y_{21}s_{22},\ 
y_{11}s_{11} + y_{21}s_{21},\\
&\quad x_{22}s_{11} + y_{21}t_{12},\ 
x_{21}s_{11} + y_{21}t_{11},\ 
x_{12}s_{11} - y_{21}t_{22},\ 
x_{11}s_{11} - y_{21}t_{21},\\
&\quad x_{21}y_{11} + x_{22}y_{12} - x_{11}y_{21} - x_{12}y_{22})
\end{split}
\end{equation}
the third ideal $I_3$ is the ideal of $Q\cap \mu^{-1}(0)$.
By using Macaulay2 we check that 
\[
x_{11}+s_{11}, x_{12}+s_{12}, x_{21}+s_{21}
\]
is a regular sequence in the ring $A=\CC[x_{11},\dots,t_{22}]/I_3$
indeed the ideal $(x_{11}+s_{11})$ is prime in $A$, $(x_{12}+s_{12})$ is prime in $A/(x_{11}+s_{11})$  and 
$(x_{21}+s_{21})$ is prime in $A/(x_{11}+s_{11}, x_{12}+s_{12})$.
\end{proof}

\begin{cor}\label{cor_fX_normal}
    We have the following isomorphism of normal varieties 
    \[
    \fX\to Q_3//\GL(2)
    \]
\end{cor}
\begin{proof}
Recall that $\fX:=F\cap\mu^{-1}(0)//\Sp(1)$ and it is normal by Corollary~\ref{cor_fXnormal}.
It is clear that the  natural map induced by the inclusion $F\to Q$ is  a bijection, and by Lemma~\ref{lem_Q3isS3} we have that $Q_3//\GL(2)$ is normal as well, hence by Zariski main theorem the morphism in the statement is an isomorphism.
\end{proof}

\begin{lem}\label{lem_lgr}
    The variety $\fX$ is the affine cone over the Lagrangian Grassmannian $\lgr(4)$.
\end{lem}
\begin{proof}
    Clearly the equation $q(X,Y)=x_{11}y_{21}-x_{21}y_{11}+x_{12}y_{22}-x_{22}y_{12}=0$ descends to the quotient $\Mat_{2\times 4}(\CC)//\Sp(1)$, the latter is described as follows:
    the affine Pl{\"u}cker embedding
    \[
    \begin{split}
        \Mat_{2\times 4}(\CC)\to &\wedge^2 V\\
        u=(X,Y)\mapsto &(x_{11}x_{22}-x_{12}x_{21}, x_{11}y_{21}-y_{11}x_{21},x_{11}y_{22}-y_{12}x_{21},\\
        &x_{12}y_{21}-y_{11}x_{22},x_{12}y_{22}-y_{12}x_{22},y_{11}y_{22}-y_{12}y_{21})
    \end{split}
    \]

    is $\Sp(1)=\SL(2)$-equivariant hence it descends to the quotient
    \[
    Pl:\Mat_{2\times 4}//\Sp(1)\to \wedge^2V.
    \]
Moreover if we restrict $Pl$ to the open subset $U=\{u|\rk u=2\}$,
\[
\Mat_{2\times 4}//\Sp(1)\supset U\to  \wedge^2 V
\]
it is injective and the image is the cone over the Grassmannian $\gr(2,4)$ in its Pl{\"u}cker embedding.

In order to prove that $Pl$ is injective we only need to prove that if $u\in \Mat_{2\times 4}$ has rank smaller or equal to one, then the closure of its orbit $\overline{\SL(2)u}$ contains $0$. This follows from the fact that 
up to the $\SL(2)$-action we can suppose that 
\[
u=
\begin{pmatrix}
    a&b&c&d\\
    0&0&0&0
\end{pmatrix}
\]
then we act with the one parameter subgroup 
$\CC^*\ni t\mapsto 
\lambda(t):=\begin{pmatrix}
    t&0\\
    0&t^{-1}
\end{pmatrix}$
and $\lambda(t)=tu$.
Using the coordinates $(p_1\dots,p_6)\in \wedge^2 V$ we get that the equation $q=0$ becomes $p_2+p_5=0$.
Hence $\fX=\im(Pl)\cap\{p_2+p_5=0\}$ that is the affine cone over the lagrangian grassmannian $\lgr(4)$.
\end{proof}

\begin{prop}\label{prop_barsigm_normal}
    The variety $\Sigmab$ is integral and normal.
\end{prop}
\begin{proof}
By \cite[Proposition 5.11]{antisymplI} the fixed locus $\fix(\bar{\tau})_{\rm red}=\Sigmab\sqcup \bar{\Omega}$ where $\Sigmab=g(\Sigma)$ and $\bar{\Omega}=g(\Omega)$ and $\Sigma$ is not contained in $\Delta$ hence $\Sigmab$ is irreducible. Moreover as it is a scheme theoretic image of a reduced scheme $\Sigmab$ is reduced.
Note that by Lemma~\ref{lem_bSigma_Delta1_empt} $\Sigma$ does not intersect $\Delta(1)$.
Recall that $\fM=\fM_1\times R$ is an analytic neighborhood of $\overline{M}$ at $x\in\Sigmab\cap c(\Delta(2))$
and an analytic neighborhood of $\bar\Sigma$ at $x$ is
    \[
    (Q\cap\mu_1^{-1}(0))\times \AAA^{4}=: Q_3\times \AAA^{4}\iso \fX\times \AAA^{4}
    \]
    where we have used Corollary~\ref{cor_fX_normal}.
    We conclude by Corollary~\ref{cor_fXnormal}.
\end{proof}

\begin{prop}\label{prop_blowup_sigma}
    The blowup $\bl_{(g(\Delta(2))\cap \Sigmab)}\Sigmab $ is isomorphic to $\Sigma$.
\end{prop}
\begin{proof}
    By the universal property of the blowup there is a unique morphism $\phi\colon \Sigma \to \bl_{(g(\Delta(2))\cap \Sigmab)}\Sigmab $ such that $\bl\circ\phi=g_|$.
    By properness $\phi$ is surjective and by construction an isomorphism on 
    \[
    \Sigma\setminus \Delta(2)\to \bl_{(g(\Delta(2))\cap \Sigmab)}\Sigmab\setminus E
    \]
    where $E$ is the exceptional divisor.
    Moreover by the local picture of the singularity of $\Sigmab$ we have that $E\to (g(\Delta(2))\cap \Sigmab)\cap \Sigmab$ is a fibration in smooth quadrics indeed the blowup of an affine cone over a $\lgr(4)$ has $\lgr(4)$ as exceptional divisor and $\lgr(4)$ is a smooth quartic 3fold.
    Hence $\phi$ restricts to a surjective morphism between fibrations in quartic 3folds. The morphism is finite since it does not contract curves and it is injective since there are no nontrivial finite morphism between quartic 3folds, see \cite[Main theorem]{quadric_finitemorhp}.
    We have thus proved that $\phi$ is a bijection and hence an isomorphism by Zariski main theorem.
\end{proof}

\subsection{The singular locus of $\mb$}
In this section we study the singularity of the moduli space $\mb$ and prove the following
\begin{prop}\label{prop_M2_Y}
    The functor $\Phi$ induces an antisymplectic involution $\tau_2$ on the smooth projective moduli space $M_2:=M_{\Sigmab}^{st}(b_2)=M_{\Sigmab}(b_2)$, with $H^2(M_2,\QQ)^{+}=\QQ \lambdab_{|M_2}$.
    Moreover $g(\fix(\tau)\cap\Delta(2))$ is a smooth cubic 4fold $Y\subset M_2$ and it is one of the two connected components of the fixed locus $\fix(M_2,\tau_2)$.
\end{prop}
Recall that the divisorial contraction $g\colon M\to \mb$ is induced by a semiample divisor $\lambda$, we will denote by $\lambdab\in \Pic(\mb)$ the Cartier divisor such that $g^*\lambdab=\lambda$, see \cite[Lemma 3.22]{antisymplI}.
We denote by $\lambda_{\sigmab,M_2}$ the line bundle on $M_2$ induced by the stability condition $\sigmab$, as defined in \cite{BM_bir}.
\begin{lem}\label{lem_M2_lambdab_rest}
The moduli space $M_2$ is a smooth projective variety and 
    the restriction of $\lambdab\in \Pic(\mb)$ to $M_2$ is the line bundle $\lambda_{\sigmab,M_2}$.
\end{lem}
\begin{proof}
    We observe that $\Delta(2)\subset M$ is closed and hence $c(\Delta(2))=M_2$ is proper. 
    This proves that $M_2=M_{\sigmab}(b_2)$ is smooth projective.
    To compute $\lambdab_{|M_2}$ consider a curve $C\subset M_2$ and denote by $i\colon M_2\to \mb$ the inclusion, by definition
    \[
    \lambdab_{|M_2}\cdot C=\lambda_{\sigmab,M_{\sigmab}(0,h,-7)}\cdot i_* C=\im (\bar{Z}(\Phi_{\kf_{i_*C}}(\ko_C)))
    \]
    where $\bar{Z}$ is the central charge of $\sigmab$ (with the assumption $\bar{Z}(0,h,-7)=-1$), $\kf$ is a quasi universal family for $M_{\sigmab}(0,h,-7)$ and $\kf_{i_*C}$ is its restriction of $\kf$ to $i(C)\times S$.
    We observe that $\kf_{|M_2\times S}\cong\kg\oplus(\ko_{M_2}\boxtimes A^{\oplus 2})$ where $\kg$ is a quasi universal family for $M_2$, here we assumed that the similitude of the universal families is one.
    We observe that
    \[
    \ko_C\otimes \kf_{i_* C}=(\ko_C\otimes \kg)\oplus(\ko_C\boxtimes A^{\oplus 2}).
    \]
    Be aware that $\lambda_{\sigmab,M_2}\cdot C$ is not $\im(\bar{Z}(\Phi_{\kg}(\ko_C)))$, indeed $\bar{Z}(b_2)\neq-1$: 
    we have 
    \[
    \lambda_{\sigmab,M_2}\cdot C=\pr^{\perp,\bar{Z}(b_2)}(\bar{Z}(\Phi_{\kg}(\ko_C)))
    \]
    where $\pr^{\perp,\bar{Z}(b_2)}$ is the orthogonal projection on $\CC=\RR^2$ on the line orthogonal to $\bar{Z}(b_2)$.
    Hence
    \[
    \lambdab_{|M_2}\cdot C=\im(\bar{Z}(\Phi_{\kg}(\ko_C)))+\im(\bar{Z}(A^{\oplus2}))
    \]
    and the observation $(0,h,-7)-2v(A)=b_2$ finishes the proof.
\end{proof}

\begin{lem}\label{lem_M2_lambda}
    Under the Mukai isomorphism  
    \[
    \vartheta_{b_2}:b_2^{\perp}\to \H^{2}(M_2,\ZZ)
    \]
    the class $(2,-h,3)$ is sent to the first Chern class of $\lambda_{\sigmab,M_2}$. In particular  $\lambda_{\sigmab,M_2}$ is a polarization of square two and divisibility two.
\end{lem}
\begin{proof}
    By \cite[Lemma 3.24] {antisymplI} the functor $\Phi$ induces an involution $\tau_2$ of $M_2$.
    Let us denote by $L=(2,-h,3),v=(0,h,-7)\in \H^{\bullet}(S,\ZZ)$, by $p,q\colon M_2\times S\to M_2, S$, we use the notation of the proof of Lemma~\ref{lem_M2_lambdab_rest}.
    Given a curve $C\subset M_2$ we have   
    \[
    \vartheta_{b_2}(L)\cdot C=p_*(\ch(\kg)^{\vee}\otimes q^*L)\cdot C={p_{|C}}_*(\ch(\kg_{C\times S})^{\vee}q^*(L))
    \]
    similarly
    \[
    \vartheta_{v}(L)\cdot i_*C={p_{|C}}_*(\ch(\kf_{i(C)\times S})^{\vee}q^*(L))
    ={p_{|C}}_*(\ch(\kg_{C\times S})^{\vee}q^*(L))+{p_{|C}}_*(q^*\ch(A^{\oplus2})^{\vee}q^*(L)).
    \]
    We deduce that
    \[
    \vartheta_v(L)\cdot i_*C={p_{|C}}_*(\ch(\kg_{C\times S})^{\vee}q^*(L))
    \]
    hence 
    $\vartheta_{b_2}(L)=
    \vartheta_{v}(L)_{|M_2}=\lambda_{\sigmab,M_2}$.
\end{proof}

\begin{lem}\label{lem_action_Phi_NS}
    Let $M_2=M(b_2)$ be the moduli space with an involution $\tau_2$ induced by an autoequvalence $\Phi$ of $\D(S)$ then for any $E\in \D(S)$ we have that $\tau_2^*\vartheta_{b_2}(\ch(E))=\vartheta_{b_2}(\ch(\Phi(E)))$.
\end{lem}
\begin{proof}
    We observe that if $\kf\in \D(M_2\times S)$ is a quasiuniversl family then we have an action $\tilde{\Phi}$ on $\D(M_2\times S)$ such that the two projections $p,q\colon\D(M_2\times
 S)\to \D(M_2),\D(S)$ are equivariant and $\tilde{\Phi}(\kf)\iso \kf$ hence the claim follows from
 \[
     \tau_2^*p_*(\ch(\kf)\cdot q^*\ch(E))=p_*(\ch(\tilde{\Phi}(\kf\otimes
     q^*E)))=p_*(\ch(\kf)\cdot q^*(\ch(\Phi(E)))).
 \]
\end{proof}

\begin{lem}\label{lem_M2_Hfix}
    The fixed part $\H^2(M_2,\QQ)^{\tau_2}=\QQ\cdot \lambda_{\sigmab,M_2}$ is generated by the Chern class of the ample divisor $\lambda_2:=\lambda_{\sigmab,M_2}$. In particular $\tau_2$ is an antisymplectic involution.
\end{lem}
\begin{proof}
    By Lemma~\ref{lem_M2_lambdab_rest} $\lambda_2=\lambda_{\sigmab,M_2}$ is the restriction of an ample line bundle hence it is ample.
    Recall that we have a $\tau,\taub$-equivariant contraction $\ml\to \mb$ of the divisor $\Delta$, where $H^2(\ml,\QQ)^{\tau}=\QQ\langle\Delta,\lambda\rangle$,
    hence $\H^2(\mb,\QQ)^{\taub}=\QQ\cdot \lambdab$ and $\H^2(M_2,\QQ)^{\tau_2}\supset\QQ\cdot \lambda_2$.
    Moreover $\NS(M_2)\iso(\ZZ(1,0,0)+\ZZ(0,0,1)+\ZZ(0,h,0))\cap b_2^{\perp}$ has rank two and $\lambda_2=(2,-h,3)$ is fixed by $\tau_2$. 
    We observe now that $\tau_2^*$ does not act as the identity on $\NS(M_2)$.
    Indeed by Lemma~\ref{lem_action_Phi_NS} the action of $\tau_2$ on $\NS(M_2)$ is induced by the action of the functor $\Phi$ on $(\ZZ(1,0,0)+\ZZ(0,0,1)+\ZZ(0,h,0))$.
    We have observed that $\Phi$ fixes $b_2=(-4,3h,-15), \lambda_2=(2,-h,3)$ hence if by contradiction $\tau_2^*$ acts as the identity on $\NS(M_2)$ then $\Phi$ acts as the identity on $(\ZZ(1,0,0)+\ZZ(0,0,1)+\ZZ(0,h,0))$ we can check that this leads to a contradiction by computing e.g. $\ch(\Phi(\ko_S))$.

We finish the proof by observing that $\tau_2^*$ acts non trivially on the trascendental part $T_2$ of $H^2(M_2,\QQ)$: we recall that $T_2$ is an irreducible Hodge substructure of $H^2(M_2,\QQ)$ for which $\H^2(M_2,\QQ)=T_2\oplus\NS_{\QQ}(M_2)$.
We observe that the holomorphic two form $\omega_2\in \H^{2,0}(M_2)$ pullsback to a nonzero holomorphic two form $g_|^*\omega_2\in \H^{2,0}(\Delta(2))=\H^{2,0}(M_2)$ and that if $i\colon \Delta(2)\to M$ is the inclusion then $i^*\omega\in H^{2,0}(\Delta(2))$ is non zero and hence a multiple of $g_|^*\omega_2$, we conclude by observing the compatibility of the actions $\tau, \tau_{|\Delta(2)}$ and $\tau_{2}$.
\end{proof}

\begin{prop}[{\cite[Theorems 1.3, 1.4]{antisymplII}}]\label{prop_M2_fixlocus}
    The fixed locus $\fix(M_2,\tau_2)$ has two irreducible components $Y,\Omega'$ where $\Omega'$ is of general type and $Y$ is a smooth cubic 4fold where ${\lambda_2}_|=\ko_Y(3)$.
\end{prop}
\begin{proof}
    It is a straightforward application of {\cite{antisymplII}[Theorems 1.3, 1.4]} once we have Lemmata~\ref{lem_M2_Hfix},~\ref{lem_M2_lambda}.
\end{proof}

In order to prove Proposition~\ref{prop_M2_Y} we need a last lemma.
We will denote by $G$ the multiplicative group with two elements, by $\rho_{\tr}$ the trivial irreducible representation and by $\rho_{\det}$ the irreducible representation acting by $-1$.

\begin{lem}\label{lem_M2_lin}
    There is no $G$-linearization $\rho$ of $(\lambda_2,\tau_2,M_2)$ such that $\rho_{|x}=\rho_{\det}$ for some $x\in \Omega'$ in the connected component of general type.
\end{lem}
\begin{proof}
    Suppose by contradiction to have such a $G$-linearization, by \cite[Lemma 3.4]{antisymplII} we have such a $G$-linearization on $(L,\mathrm{inv}, M_h(0,h,-3))$. 
    By \cite[Proof of Proposition 3.5]{antisymplII} we have that 
    \[
\H^0(M_h(0,h,-3),L)^{\mathrm{inv}}=\H^0(M_h(0,h,-3),L),
    \]
    hence $\Omega\subset M_h(0,h,-3)$ is contained in the base locus of $L$.
    We also recall that $L:=\pi^*\ko_{\PP^4}(1)\otimes\ko(\Delta)$ and $\Omega$ is not contained in $\Delta$, see \cite[Lemma 5.8]{antisymplI}, hence for any $x\in \Omega\setminus \Delta$ there is a global section $s$ of $L$ such that $s(x)\neq 0$.
\end{proof}

\begin{proof}[Proof of Proposition~\ref{prop_M2_Y}]
    By Lemma~\ref{lem_M2_Hfix} $\tau_2$ is an antisymplectic involution with $$H^2(M_2,\QQ)^{+}=\QQ \lambdab_{|M_2}.$$
    moreover by Lemma~\ref{lem_M2_lambda} we have that $\lambdab_{|M_2}$ is a polarization of square two and divisibility two on a smooth hyperk{\"a}hler of $K3^{[4]}$ deformation type.
    By \cite[Proposition B.12]{Deb_Mac_HK} $M_2$ is isomorphic to a LLSvS variety associated to a smooth cubic 4fold $Y\subset \PP^5$ containing no planes. 
    By \cite[Example 1.2]{antisymplII} $\fix(\tau_2)$ contains $Y$ as a connected component, moreover as observed already in Proposition~\ref{prop_M2_fixlocus} $\fix(\tau_2)=Y\sqcup \Omega'$ has two connected components and $\Omega'$ is of general type.
     We know that the image 
     $g(\Delta(2)\cap \fix(\tau))$ has dimension four and it is irreducible, hence we are only left to prove that 
     $g(\Delta(2)\cap \fix(\tau))\subseteq Y$, indeed they are both irreducible of the same dimension.
     In order to this we consider the $G:=\ZZ/2\ZZ$-linearization $\rho$ of $(\lambda,\tau,\ml)$ and $\rhob$ of $(\lambdab,\taub,\mb)$ defined in \cite[Theorem 5.12, Proof of Proposition 5.11]{antisymplI}, recall also that the contraction $g\colon\ml\to \mb$ is equivariant, $g^*\lambdab=\lambda$ and 
     \[
     \begin{split}
         \rho_{|x}&=\rho_{\det} \text{ for } x\in \Sigma\\
         \rho_{|x}&=\rho_{\tr}\text{ for }x\in \Omega,
     \end{split}
     \]
     We observe that $\rhob$ restricts to a $G$-linearization of $(\lambda_2,\tau_2,M_2)$ and that for $x\in g(\Delta(2)\cap \Sigma)$ we have ${\rho_2}_{|x}=\rho_{\det}$.
     By Lemma~\ref{lem_M2_lin} we know that $x$ can not be a point in the component of $\fix(\tau_2)$ of general type.
\end{proof}

\section{The semi stable family}\label{section_semistable}
In this section we construct a flat family over a disc where the general fibre is a the Fano component of the fixed locus $\fix(\tau)$ and the special fibre is a transverse union of $\Sigma$  and a fibration in smooth quadric 4folds over a smooth cubic 4fold.
The main result is the following
\begin{thm}[Proposition~\ref{prop_bl_fundamental_family}]\label{thm_bl_fundamental_family}
    There exists a semistable family over the disc $\kf\to \DDD$ such that 
    \begin{itemize}
        \item for $t\neq 0$ the fibre $\kf_t$ is deformation equivalent to the Fano variety $F$ defined in the introduction,
        \item the central fibre is a transverse union $\kf_0=\Sigma\cup (Q_4/Y)$ and $\Sigma\cap (Q_4/Y)=Q_3/Y$, 
    \end{itemize}
     where $Q_4/Y$ (resp. $Q_3/Y$) is a fibration in smooth quardic 4folds (resp. 3folds) bundle over a cubic 4fold $Y$.
\end{thm}

\subsection{Smoothing of $\overline{\Sigma}$}
Let us recall 
some results of Markman and Namikawa formulated in \cite[Section 2]{antisymplI}.
Let us consider the divisorial contraction $g\colon M\to \mb$, we know that:
\begin{itemize}
    \item the deformation spaces $\defo(M),\defo(\overline{M})$ are smooth, \cite[Theorem 2.11]{namikawa},
    \item there is a natural $2:1$ cover $m:\defo(M)\to \defo(\overline{M})$, branched over $B\subset \defo(\overline{M})$, \cite[Theorem 1.4]{markman10a},
    \item the deformation space of the pair $\defo(\overline{M},\bar{\tau})\subset \defo(\overline{M})$ is a smooth subvariety of codimension one, not contained in $B$, \cite[Proposition 2.1]{antisymplI}.
\end{itemize} 
By choosing a general disc $\DDD\subset \defo(\overline{M},\bar{\tau})$ we get a smoothing 
\[
\MMb\to \DDD
\]
of $(\overline{M},\bar{\tau})$.

Note also that we have an involution $\tau_{\MMb}$ on $\MMb$ preserving the fibres and by 
\cite[Corollary 2.2]{antisymplI} 
there is a  relatively ample line bundle $\klb$ on $\MMb$ such that
\begin{itemize}
    \item $c_1(\klb_0)=\bar{\lambda}$
    \item for any $t\in \DDD$ we have $\H^2(\MMb_t,\QQ)^{+,\tau_{\MMb,t}}=\QQ\cdot c_1(\klb_t)$
    \item for $t\neq 0$ we have $q_{\MMb_t}(c_1(\klb_t))=2$ and $\mathrm{div}(c_1(\klb_t))=2$.
\end{itemize}
\begin{lem}
    The morphism $\MMb\setminus \overline{M}_{\sing}\to \DDD$ is smooth, in particular $\MMb_{\sing}\subset  \overline{M}_{\sing}$ and $\MMb$ is normal.
\end{lem}
\begin{proof}
    The family $\MMb\to \DDD$ is flat and the  fibres of the restricted morphism are smooth. This implies that $\MMb\setminus \overline{M}_{\sing}$ is regular and hence smooth.
    The last statement follows from \cite[Lemma 3.10]{antisymplII}.
\end{proof}
Consider $(\klb,\MMb)$ and the natural $\mu_2$ linearization on $\klb_0$,
see \cite[Proof of Proposition 5.11]{antisymplI}.
By \cite[Lemma 3.4]{antisymplII} the $\mu_2$ linearization deforms on the family $(\mathcal{L},\MM)$.
    Define
    \[
    V_*:=\fix(\tau_{\MM})\setminus (\Sigmab_{\sing}\cup\overline{\Omega}_{\sing})
    \]
    and
    $V:=\overline{V_*}=\fix(\tau_{\MM})_{\rm red}$ the closure in $\MMb$.
    Then define
    \[
    \mathcal{Y}_*^{\pm}:=\{x\in V_*:{\mu_2}_{|\mathcal{L}(x)}=\CC^{\pm}\}.
    \]
    and
    \[
\mathcal{Y}^{\pm}:=\overline{\mathcal{Y}_*^{\pm}}\subset \MMb
    \]
    the closure of $\ky_*^{+/-}$ in $\MMb$.
    We observe that
    \begin{itemize}
    \item  $\mathcal{Y}^{+/-}\to \DDD$ is flat,
    \item $V=\mathcal{Y}^+\sqcup\mathcal{Y}^-$, see \cite[Proof of the Main Theorem, p.40]{antisymplI},
    \item $\mathcal{Y}^-_{0,\red}=\overline{\Sigma}$.
    \end{itemize}

\begin{lem}\label{lem_smoothing_Sigmab}
    The projective flat family $\mathcal{Y}^-\to \DDD$ has normal total space and the central fibre $\mathcal{Y}^-_0=\overline{\Sigma}$ is reduced.
\end{lem}
\begin{proof}
By construction we have that $\mathcal{Y}^-\setminus \Sigmab_{\sing}\to \DDD$ is smooth.
We apply \cite[Lemma 3.10]{antisymplII} by observing that $\mathcal{Y}^-_{0,\red}=\Sigmab$ is integral and normal by  Proposition~\ref{prop_barsigm_normal}.
\end{proof}
We expect $\ky^-$ to be a regular, however we were not able to prove this: we will overcome this issue by studying the deformation of its singular fibre $\ky^-_0=\Sigmab$.
We will denote by $\rm{T}^i=\Extsh^i(\Omega_{\Sigmab},\ko_{\Sigmab})$ the ext sheaves. 
\begin{lem}
    The singular Fano variety $\Sigmab$ is 3-du Bois and hence have unobstructed deformations.
\end{lem}
\begin{proof}
    By \cite[Theorem 4.5]{friedman_laza} a Fano variety that is l.c.i. and is 1-du Bois has unobstructed deformations.
    Hence we are left to prove that it is 3-du Bois, see \cite[Remark 2.11]{friedman_laza} for a comment on the definition.
    We recall that an analytical neighborhood $U$ of $x\in \Sigmab_{\sing}$ is a product $\CC^4\times \{q_2=0\}$ where $q_2=x_0^2+\cdots+x_4^2$ hence any singular point is ordinary point of multiplicity 2 and we deduce from \cite[Theorem E]{MP_hodgeideals} that the local minimal exponent $\tilde{\alpha}_x(U)=4$ hence $\tilde{\alpha}(U)=4$.
    By \cite[Theorem 1.1]{MOPW} the natural morphism $\Omega^i\to \underline{\Omega}^i$ to the graded part of the filtered de Rham complex is an isomorphsm for $i=0,1,2,3$ hence $\Sigmab$ is 3-du Bois.
\end{proof}

\begin{prop}\label{prop_defo_Sigmab}
    For any disc $\DDD\subset\defo(\Sigmab)$ if the induced deformation $\kx\to \DDD$ is not locally trivial then $\kx$ is regular.
\end{prop}
\begin{proof}
Let us compute the sheaf $\rm{T}^1$, it is clear that it is supported on the singular locus, moreover an analytical neighborhood $U$ of $x\in \Sigmab_{\sing}$ is a product $\CC^4\times \{q_2=0\}$ where $q_2=x_0^2+\cdots+x_4^2$.
Hence $\rm{T}^1(U)$ is the Jacobian ring of the singularity and we conclude that $\rm{T}^1=\ko_{\Sigmab_{\sing}}$. 
By Proposition~\ref{prop_M2_Y} we get $\rm{T}^1=\ko_Y$ where $Y=\Sigmab_{\sing}$ is a smooth cubic 4fold.
    Let us recall the exact sequence induced by the spectral sequence for Ext
    \begin{equation}\label{eq_deformation}
        0\to \H^1(\Sigmab,\rm{T}^0)\to \Ext^1(\Omega_{\Sigmab},\ko_{\Sigmab})\overset{\phi}{\to} \H^0(\Sigmab,\rm{T}^1)=\CC
    \end{equation}
    note also that $\H^0(\Sigmab,\rm{T}^1)$ is generated by the canonical surjection $\ko_{\Sigmab}\epi \ko_Y$.
    For a deformation $f\colon\kx\to \DDD$ over a disc $\DDD\subset \defo(\Sigmab)$ and its class $[f]\in \Ext^1(\Omega_{\Sigmab},\ko_{\Sigmab})$, we have that $f$ is locally trivial if and only if $[f]\in \H^1(\Sigmab,\rm{T}^0)$, see \cite[Section 2]{friedman}.
    By the proof of \cite[Proposition 2.5]{friedman} we have that if 
    $\phi(f)\colon \ko_{\Sigmab}\to \ko_{Y}$ is surjective then $\kx$ is regular.
    We also observe that the existence of the smoothing in Lemma~\ref{lem_smoothing_Sigmab} shows that $\phi$ is non trivial hence surjective.
\end{proof}

\begin{cor}\label{cor_regular_family}
    There exists a flat family with regular total space $\kx\to \DDD$ whose smooth fibres are deformation equivalent to $\ky^-_t, t\neq 0$ and $\kx_0=\Sigmab$.
\end{cor}
\begin{proof}
    Let us consider the family $\ky^-\to \DDD$ and the induced non constant morphism $\nu\colon \DDD\to \defo(\Sigmab)$. 
    This proves that for any disc $\DDD\subset \defo(\Sigmab)$ intersecting $\nu(\DDD)\setminus \{0\}$ the induced family $\kx\to \DDD$ is not locally trivial, up to shrinking it we can assume that $\kx_t,t\neq 0$ are smooth and they are deformation equivalent to $\ky^-_t,t\neq 0$.
    Moreover by Proposition~\ref{prop_defo_Sigmab} the total space $\kx$ is regular.
\end{proof}
We use the following general result
\begin{lem}\label{lem_localstr}
    Let $\kh \to \DDD$ a flat family with reduced central fibre $\kh_0$ that has a hypersurface singularity $0\in \kh_0$ and such all the other fibres are smooth.
    Then, in an analytic neighborhood of $0\in \kh$, the total space $\kh $ is a hypersurface in a smooth ambient variety with equation of the form
    \[
    \{ h_0+th_1+\cdots=0\}
    \subseteq \CC_x^{m+1}\times \CC_t^1\to \CC^1_t
    \]
where $h_i\in \CC[x_0,\dots,x_m]$ and the morphism is the projection.
\end{lem}
\begin{proof}
We observe that $\kh\setminus\{0\}\to \DDD$ is a smooth family and $\kh_0$ is a locally complete intersection, hence by \cite[Syntomic morphisms, Lemma 29.30.10]{stack_proj} we have that in some affine neighborhoods the family $h$ is
\begin{equation}
    h: \kh=\Spec(\CC[x_0,\dots,x_n,t]/(f_1,\dots, f_{d}))\to \Spec(\CC[t]).
\end{equation}
We know that $f_1(x,t=0),\dots, f_{d}(x,t=0)$ in $\Spec(\CC[x_0,\dots,x_n])$ defines the central fibre, which has a singularity at $0$.
We use the notation $F=(f_1,\dots, f_{d})$ and consider  the jacobian
\begin{equation}
\Jac F(x,t)_{(0,0)}
=
    \begin{pmatrix}
        \partial_{x_0}f_1&\cdots & \partial_{x_n}f_1&\partial_t f_1\\
\cdots&\cdots&\cdots&\cdots\\
        \partial_{x_0}f_{d}&\cdots & \partial_{x_n}f_{d}&\partial_t f_{d}\\
    \end{pmatrix}
\end{equation}
and we observe that the matrix obtained by removing the right most column is the Jacobian $\Jac F(x,0)_0$ of the equations defining the central fibre.
Since the central fibre has a hypersurface singularity at $0$, the tangent space has dimension $n+1-d$ and $\Jac F(x,0)_0$ has rank $d-1$.
We conclude that $\Jac F(x,t)_{(0,0)}$ has either rank $d-1$ or $d$.
In the latter case the total space $\kh$ is regular.
In both cases we can assume that $y_0:=f_1,\dots, y_{d-2}:=f_{d-1}$ are local analytic coordinates defining a smooth variety
and the analytic equation for $\kh$ become
\[
y_0=\cdots=y_{d-2}=0, \; \bar{f}(y_0,\dots,y_{d-2},x_{d-1},\dots, x_n,t)=0
\]
where 
$\bar{f}(y_0,\dots,y_{d-2},x_{d-1},\dots, x_n,t):=f_{d}(x_0(y_0,\dots,y_{d-2}),\dots, x_{d-2}(y_0,\dots,y_{{d-2}}),x_{d-1},\dots, x_n,t)$.
We conclude that
$\kh$ is defined, in an analytical neighborhood, as the zero set of \[
f(x_{d},\dots, x_n,t):=\bar{f}(x_0(0),\dots,x_{d}(0),x_{d-1},\dots, x_n,t)
\]
in $\CC^{n-d+1}_x\times \CC_t$.
\end{proof}
By Lemma~\ref{lem_localstr} we have that $\kx$ is, locally around a singular point of the central fibre $x=0\in \Sigmab=\kx_0$, a regular hypersurface defined by the equation
\begin{equation*}
    f=q+th_1+t^2h_1+\cdots+t^lh_l
\end{equation*}
where $q,h_i\in \CC[x_0,\dots, x_8]$.
Let us observe that we can choose the equation of a better form indeed we know that $\kx$ is regular and $Y\subset \kx$ is a smooth subvariety passing trought $x$ hence locally analytically 
\[\kx= V\times U\] 
where $V$ is a neighborhood of $x$ in $Y$ we consider the closed embedding
\[
\begin{split}
    V\times U\mono \CC_x^9\times \CC_t^1\\
    V\times 0\mono \CC_x^9\times 0
\end{split}
\]
hence up to changing coordinates on $\CC^9_x$ we can suppose that 
the second embedding factors as the following local isomorphism
\[
\rho_1\colon V\times 0\mono \CC_{x_5,x_6,x_7,x_8}\times 0,
\]
we denote by $V'\subset \CC_{x_5,x_6,x_7,x_8}$ the image.
Similarly we consider the closed embedding
\[
\rho_2\colon 0\times U\mono 0\times \CC_{x_1,x_2,x_3,x_4}\times \CC_t
\]
and denote by $U'\subset \CC_{x_1,x_2,x_3,x_4}\times \CC_t$ the image.
We obtained the following commutative diagram
\begin{equation}
    \begin{tikzcd}
    V\times U\ar[r,"{(\rho_1,\rho_2)}"]\ar[dr]&V'\times U'\ar[d]\\
    {}&\CC_{x_5,\dots,x_8}\times\CC_{x_1,\dots,x_4}\times \CC_t\\
\end{tikzcd}
\end{equation}
we know write $f$ in the new coordinates and we observe that
$f(x,t)=0$ is equivalent to\\
$f(x_1\dots, x_4,0\dots,0,t)=0$ hence
\begin{equation}\label{eq_defX}
    f=q+th+\text{ higher order terms}
\end{equation}
where $q=x_0^2+x_1^2+x_2^2+x_3^2+x_4^2$ and $h\in \CC[x_0\dots, x_4]$.
Moreover we know that $\kx$ is regular hence $0\neq \nabla f_0=(\nabla q_0,h(0))=(0,h(0))$ and  $h(0)\neq 0$.

We will denote by $\ky'$ the following base change
\begin{equation}
    \begin{tikzcd}
        \ky'\ar[r]\ar[d]&\kx\ar[d]\\
        \DDD\ar[r,"t^2"]&\DDD
    \end{tikzcd}
\end{equation}

\begin{prop}\label{prop_family_Y'}
    The flat family $\ky'\to \DDD$ satisfies the following properties
    \begin{itemize}
        \item for $t\neq 0$ the fibre $\ky'_t$ is deformation equivalent to the Fano connected of component $\fix(\MMb,\taub_{t'})$ for some $t'\neq 0$,
        \item the total space $\ky'$ is integral and normal, more precisely the singularity is locally analytically $\CC^4\times \ks$ where $\ks\to \DDD$ has local equations 
\begin{equation}
    s: \Spec(\CC[x_0,\dots,x_4,t]/(q+t^2h+\text{higher order terms})))\to \Spec(\CC[t]),
\end{equation}
where $h\in \CC[x_0,\dots,x_4]$ with $h(0)\neq 0$  and $q\in \CC[x_0,\dots, x_4]_2$ is a quadric of maximal rank,
        \item the central fibre is $\ky'_0=\overline{\Sigma}$.
    \end{itemize}
\end{prop}
\begin{proof}
    We observe that the fibres of a base change remains the same hence the first and the third points in the claim follows by the construction of $\ky'$.
    The second claim follows from the expression in~\eqref{eq_defX}.
\end{proof}

Recall that by Theorem~\ref{thm_contr_g} $g(\Delta\cap \Sigma)=g(\Delta(2))=Y\subset \Sigmab$ where $Y$ is a smooth cubic 4fold.

\begin{prop}\label{prop_bl_fundamental_family}
    The blow-up 
    \begin{equation}
        \kf:=\bl_{Y}\ky'\to\DDD
    \end{equation}
    is a flat family  with regular total space, reduced central fibre
    $E\cup \Sigma$, where $E$ is a fibration in quadric 4folds over $Y\subset\Sigmab$, moreover the intersection $E\cap \Sigma$ is transverse and it is a fibration in quadric 3folds over $Y$.
\end{prop}
\begin{proof}
    Consider the blow-up morphism $\bl_{Y}\ky'\to \ky'$ and let us denote by $E$ the exceptional divisor. 
    A neighborhood of a point in $Y$ has the form $\CC^4\times \ks$ and the singular locus is $\CC^4\times\{\pt\}$ where $\pt\in \ks$ is the singular point of the special fibre $q=0$.
    Note that $\CC^4$ in this product is a chart of the 4fold $Y$, recall that we are considering neighborhoods in the analytic topology, so $\CC^4$ is an abuse of notation for an analytic neighborhood of $0\in \CC^4$.
    Hence the blowup locally splits as 
    \[
    \CC^4\times\bl_{\pt}\ks.
    \]
    By local computations the blowup is regular and the exceptional divisor $E_{\pt}$ of $\bl_{\pt}\ks$ is the quadric $\{q+t^2=0\}\subset \PP^5$. 
    Let us denote by $Q_a:=\{q=0\}\subset \CC^5$ the affine cone over $Q_p:=\{q=0\}\subset\PP^4$.
    The zero fibre $\bl_{\pt}\ks_0$ is the union of the strict transform of $Q_a$, that is $\bl_{\pt}Q_a$, and the exceptional divisor $E_{\pt}$.
    By local computations we also see that $E_{\pt}\cap \bl_{\pt}Q_a=Q_p$. 
    We conclude by applying Proposition~\ref{prop_blowup_sigma} saying that the strict transform of $\Sigmab$ is $\Sigma$.
\end{proof}

\begin{rmk}
    Note that the procedure above is very similar to the simultaneous resolution of a family of cubic 4folds with an ODP considered in \cite{kuz_simult_cat_res}.
\end{rmk}

\section{Hodge numbers}\label{section_CS}
In this section we compute the Hodge numbers of the Fano eightfold as announced in the introduction.
We prove that the semistable degeneration construdcted in the previous section has no modnodromy hence the motivic nearby cycle computes the Hodge-Deligne polynomial of the smooth fibre.

We start by computing the class of $\Sigma$, defined in~\eqref{eq_diag_fixedloc}, and the class of a 
fibration in smooth quadric 4folds (resp. 3folds)
over a cubic 4fold $Y$ denoted by $Q_4/Y$ (resp. $Q_3/Y$) in the Grothendieck ring of varieties.
\begin{lem}\label{lem_groth_class}
    In the Grothendieck ring of varieties we have
    \begin{equation}
        \begin{split}
            [\Sigma]&=\sum_{i=0}^8\LL^i
+\sum_{i=1}^5\LL^i\cdot [S]+\LL^2\cdot[S^{[2]}]\\
[Q_3/Y]&=[Y]\cdot\sum_{i=0}^{3}\LL^i\\
[Q_4/Y]&=(1+\LL+2\LL^2+\LL^3+\LL^4)[Y]
\end{split}
    \end{equation}
\end{lem}
\begin{proof}
    By definition $\Sigma$ is the Mukai flop of $\bl_S\PP^8$ where the base of the bundles is $S^{[2]}$ and the bundles are of rank $2$ and $3$ respectively.
    Hence
    \[
    [\Sigma]-[S^{[2]}][\PP^2]=[\bl_S\PP^8]-[S^{[2]}][\PP^1]
    \]
    gives the first equality.
    Moreover we have 
    \[
    \begin{split}
        [Q_3/Y]&=[Q_3][Y]\\
[Q_4/Y]&=[Q_4][Y].
    \end{split}
    \]
    We observe that $Q_4$ is a smooth quadric of dimension 4 and $Q_3$ is a smooth quadric of dimension 3.
    If $Q_n\subset\PP^{n+1}$ is a smooth quadric of dimension $n$ we can use the projection from a point to get the relation $[Q_n]=1+\LL [Q_{n-2}]+\LL^{n}$, this concludes the proof.
\end{proof}

\begin{cor}
The varieties $\Sigma,Q_4/Y,Q_3/Y$ have zero odd cohomology and the even part of the Hodge diamonds are
\begin{itemize}
    \item the Hodge diamond of $\Sigma$ is
    \begin{equation}
    \begin{tikzcd}[sep=tiny]
        0&0&1&22&255&22&1&0&0\\
        &0&0&2&44&2&0&0&\\
        &&0&1&23&1&0&&\\
        &&&0&2&0&&&\\
        &&&&1,&&&&\\
    \end{tikzcd}
    \end{equation}
    \item the Hodge diamond of $Q_4/Y$ is
    \begin{equation}
    \begin{tikzcd}[sep=tiny]
        0&0&0&2&46&2&0&0&0\\
        &0&0&1&25&1&0&0&\\
        &&0&1&24&1&0&&\\
        &&&0&2&0&&&\\
        &&&&1,&&&&\\
    \end{tikzcd}
\end{equation}
\item the Hodge diamond of $Q_3/Y$ is
\begin{equation}
    \begin{tikzcd}[sep=tiny]
        &0&0&1&24&1&0&0&\\
        &&0&1&23&1&0&&\\
        &&&0&2&0&&&\\
        &&&&1.&&&&\\
    \end{tikzcd}
\end{equation}
\end{itemize}
\end{cor}
\begin{proof}
We compute Hodge-Deligne polynomials $HD$ using the fact that it is a multiplicative Euler-Poincaré characteristic 
\[
\begin{split}
    HD(-):\Var&\to \ZZ[u,v]\\
    X&\mapsto\sum \h^{p,q}(X)u^pv^q\text{ for }X \text{ smooth projective}.
\end{split}
\]
We apply Lemma \ref{lem_groth_class} and the knowledge of the Hodge numbers of $Y,S,S^{[2]}$.
\end{proof}

We will now use Clemens-Schmid's theory to compute the Hodge numbers of the Fano eightfold, see \cite{morrison_clemens_schmid} for a survey.
We use use the notation 
\[
\begin{split}
    X&=X_1\cup X_2\\
    X_1&=\Sigma\\
    X_2&=Q_4/Y\\
    X^{[0]}&=X_1\sqcup X_2\\
    X^{[1]}&=X_1\cap X_2=Q_3/Y\\
    X^{[p]}&=\varnothing \text{ for }p\neq 0,1.
\end{split}
\]
We now compute the Clemens-Schmid spectral sequence
\[
E_1^{p,q}:=\H^q(X^{[p]},\QQ)\implies \H^{\bullet}(X,\QQ)
\]
which degenerate at $E_2$, see \cite{morrison_clemens_schmid}.

\begin{lem}\label{lem_surj}
    The restriction map $\H^i(Q_4/Y,\QQ)\to \H^i(Q_3/Y,\QQ)$ is surjective for all $i\geq 0$.
\end{lem}
\begin{proof}
It is enough to prove that the restriction map $\H^i(Q_4,\QQ)\to \H^i(Q_3,\QQ)$ is surjective for all $i$ even.
The last one follows from Lefschetz hyperplane theorem and the computation of the cohomology.
\end{proof}
The differentials
\[
d_1^{p,q}:E_1^{p,q}\to E_1^{p+1,q}
\]
are restrictions and they are surjective by Lemma~\ref{lem_surj}. 
The odd cohomology is zero hence the spectral sequence is
\begin{equation*}
    \begin{tikzcd}
        0&0&0&0\\
        0&\H^{q}(X^{[0]},\QQ)\ar[r]&\H^{q}(X^{[1]},\QQ)&0\\
        0&0&0&0
    \end{tikzcd}
\end{equation*}
for $q$ even.
Hence we obtain
\begin{equation}\label{eq_spectr_X}
    \begin{split}
        \H^{k}(X,\QQ)&
        =\bigoplus_{p+q=k}E_2^{p,q}\\
        &=\bigoplus_{p+q=k}\ker(\H^{q}(X^{[p]},\QQ)\to\H^{q}(X^{[p+1]},\QQ))=\ker(\H^{k}(X^{[0]},\QQ)\to\H^{k}(X^{[1]},\QQ))\\
        &=E^{0,k}_2
    \end{split}
\end{equation}

for $k$ even and $\H^{\textrm{odd}}(X,\QQ)=0$.
Moreover the \textit{weight filtration} on
\[
0\subseteq W_0(\H^{k}(X,\QQ))\subseteq  W_1(\H^{k}(X,\QQ))\subseteq\cdots\subseteq W_{k}(\H^{k}(X,\QQ))=\H^{k}(X,\QQ)
\]
induced by 
\[
W_k=\bigoplus_{q\leq k}E^{\bullet,q}
\]
is trivial, indeed the grading is $\kg_j(\H^{k}(X,\QQ))=E_2^{k-j,j},j=0,\dots,k$
and $E_2^{p,q}=0$ for $p\neq 0$.

Let us recall the \textit{monodromy weight filtration} on the cohomology of a smooth fibre of the family $\kf\to \DDD$. 
We will use the same notation as in \cite{morrison_clemens_schmid} hence
\[
\H^{\bullet}_{\lim}:=\H^{\bullet}(\kf_t,\QQ)
\]
for some $\DDD\ni t\neq 0$.
The Picard Lefschetz transformation is defined to be the linear map
\[
T: \H^{k}_{\lim}\to \H^{k}_{\lim}
\]
induced by a generator of $\pi_1(\DDD^*)$.
By \cite{landman} the map $T$ is unipotent, more precisely $(T+\id)^{k+1}=0$.
Hence we can define its logarithm
\[
N=\ln(T)=(T-\id)-\frac{1}{2}(T-\id)^2+\frac{1}{3}(T-\id)^3-\cdots.
\]
We observe that $N$ is nilpotent with index of nilpotency equal to the index of unipotency of $T$, in particular $N^{k+1}=0$. 
The monodromy weight filtration 
\[
0\subseteq W_0\subseteq\cdots\subseteq W_{2k}=\H^{k}_{\lim}
\]
is defined in the following way:
\begin{itemize}
    \item set $W_0:=\im N^k$ and $W_{2k-1}:=\ker N^{k}$, since $N^{k+1}=0$ we have $W_0\subseteq W_{2k-1}$ and the induced map 
    \[
    N: W_{2k-1}/W_0\to W_{2k-1}/W_0
    \]
    satisfies $N^k=0$,
    \item define 
    \begin{equation*}
        \begin{split}
            W_1=\text{preimage of }(\im N^{k-1}\subseteq W_{2k-1}/W_0)\subseteq W_{2k-1}\\
            W_{2k-2}=\text{preimage of }(\ker N^{k-1}\subseteq W_{2k-1}/W_0)\subseteq W_{2k-1}
        \end{split}
    \end{equation*}
    observe that $W_0\subseteq W_1\subseteq W_{2m-2}\subseteq W_{2m-1}$ and the induced map
    \[
    N: W_{2k-2}/W_1\to W_{2k-2}/W_1
    \]
    satisfies $N^{k-1}=0$
    \item continue for $l=1,\dots, k-1$.
\end{itemize}

By results of Bittner, Clemens, Griffiths, Steenbrink and Schmid, see \cite[Clemens-Schmid I, Corollary 2]{morrison_clemens_schmid} and \cite[Theorem 7.1.3, Corollay 7.2.3]{peters}, we have the following
\begin{prop}
For a semistable degenration as above we have that
\begin{itemize}
    \item if the weight filtration on $\H^k(X,\QQ)$ is trivial then the monomdromy $N=0$,
    \item if the monodromy $N$ is zero then the rational Hodge structure of $\H^{\bullet}_{\lim}$ is the rational Hodge structure of the motivic nearby cycle
    \[
    \psi_f^{\mathrm{mot}}:=[X_1]+[X_2]-(1+\LL)[X_1\cap X_2]
    \]
    when we see them in the Grothendieck group of rational mixed Hodge structures.
\end{itemize}
\end{prop}

We observed in \eqref{eq_spectr_X} that the weight filtration on 
$\H^k(X,\QQ)$ is trivial hence the monodromy is zero and the Hodge-Deligne polynomial of $\kf_t$ is
\[
HD(\kf_t)=HD(X_1)+HD(X_2)-(1+uv)HD(X_1\cap X_2).
\]
\begin{cor}
    The odd cohomology of $\kf_t,t\neq 0$ is zero and the Hodge diamond is
    \begin{equation*}
    \begin{tikzcd}[sep=tiny]
        0&0&1&22&253&22&1&0&0\\
        &0&0&1&22&1&0&0&\\
        &&0&1&22&1&0&&\\
        &&&0&1&0&&&\\
        &&&&1.&&&&\\
    \end{tikzcd}
\end{equation*}
\end{cor}

\end{document}